\documentclass[12pt]{amsart}  

\usepackage{custom_preamble} 

\begin{document}

\title{The structure theory of set addition revisited}

\author{\tsname}
\address{\tsaddress}
\email{\tsemail}

\begin{abstract}
In this article we survey some of the recent developments in the structure theory of set addition.
\end{abstract}

\maketitle

\section{Introduction}

The purpose of this survey is to review some recent advances in Fre{\u\i}man's theorem, one of the central results in what is called the structure theory of set addition.  This theory was first systematically developed by Fre{\u\i}man in \cite{fre::,fre::0} and a large part of it is concerned with the question `what do approximate groups look like?'

In fact we shall be interested in what \emph{Abelian} approximate \emph{cosets of sub-}groups look like.  To craft a more concrete question it is useful to have some notation: suppose, as it shall be throughout, that $G$ is an Abelian group.  Given $A,A' \subset G$ we write $A+A'$ for the \textbf{sumset} of $A$ and $A'$ which is defined by
\begin{equation*}
A+A':=\{a+a': a \in A, a' \in A'\}.
\end{equation*}
Given this it is easy to check that a subset $W$ of $G$ is a coset (of a subgroup) if and only if
\begin{equation*}
W \neq \emptyset \text{ and } |W+W|=|W|.
\end{equation*}
As indicated we are interested in approximate cosets and to this end we relax these requirements so that they are only approximately true.  Relaxing the first requirement does not lead to an interesting generalisation; for the second we ask that the sumset be `not much larger' than the original set.  To be clear given $K \geq 1$ we say that $A$ (non-empty) has \textbf{doubling}\footnote{One might very reasonably suggest that one use the phrase `doubling ratio' instead of `doubling' here.  While this would be sensible, this is not the terminology in use in the subject and to maintain consistency with existing literature we shall follow the standard terminology.} $K$ if $|A+A| \leq K|A|$ and are interested in which sets have this property.

We shall be interested in the case when the doubling is small and to get a sense of what this means it is worth noting that trivially \emph{every} set has doubling $|A|$ since there cannot be more elements in $A+A$ than there are pairs in $A^2$.  (In fact this can trivially be improved to $(|A|-1)/2$ but our interest at this stage is really in orders of magnitude.)

It may be instructive on a first read to think of $K=O(1)$ as $|A| \rightarrow \infty$, although it will turn out later that we can allow $K$ to grow (slowly) with $|A|$.

If $A$ is a coset then $A$ has doubling $1$ which is certainly small, but are there any other sets with small doubling?  One way to create such sets is to take a large subset of a coset.  In particular, suppose that $W$ is a coset in $G$ and $A \subset W$ is such that $|W| \leq K|A|$.  Then since $A+A \subset W+W$ we conclude that $A$ has doubling $K$. 

It turns out that if $K$ is small enough then the above construction is characteristic in the sense that it is the \emph{only} way to create sets with doubling $K$.  This was proved by Fre{\u\i}man in \cite{fre::3} and appears as \cite[Exercise 2.6.5]{taovu::}.  At this point it is worth remarking that the book \cite{taovu::} of Tao and Vu is the standard text for many of the better known aspects of the material we shall be discussing and, where possible, we have given references to that alongside the original source.
\begin{proposition}\label{prop.simpfrei}
Suppose that $|A+A| \leq K|A|$ for some $K<1.5$.  Then there is a subgroup $H$ of size at most $K|A|$ such that $A$ is contained in a coset of $H$.
\end{proposition}
(For the unfamiliar it may be worth saying that this result is \emph{not} the Fre{\u\i}man's theorem we shall ultimately be interested in.)

The result shows that the only way of creating sets with doubling less than $1.5$ is the method described before the proposition and, moreover, every set created in that way has doubling less than $1.5$: the result characterises sets with doubling less than $1.5$. 

There is a good reason for the limitation of $1.5$ above and that is because there is a qualitatively new way of constructing sets with small doubling.  Suppose that $H \leq G$, $x+H \in G/H$ has order $4$ and put $A:=H \cup (x+H)$.  Then a short calculation shows that $|A+A| = 1.5|A|$ but any coset containing $A$ has size at least $2|A|$.

Instead of taking large subsets of one coset our new construction takes unions of cosets (of the same subgroup).  In light of this we introduce a new piece of terminology\footnote{This follows Green and Ruzsa \cite{greruz::}, and has been much popularised by Tao \cite{tao::6}.}: we say that a set $A$ is \textbf{$k$-covered} by $B$ if there is a set $X$ of size at most $k$ such that $A \subset X+B$. 

One can combine the two ways of creating sets with small doubling by considering (disjoint) unions of large subsets of cosets (of the same subgroup) to produce more sets with small doubling, and it turns out that while the doubling remains less than $2$ this is the only way of creating sets with doubling less than $2$.
\begin{proposition}\label{prop.simkneser}
Suppose that $|A+A| \leq K|A|$ for some $K<2-\epsilon$.  Then there is a subgroup $H$ of $G$ such that $|H| \leq K|A|$ and $A$ is $O_\epsilon(1)$-covered by $H$.
\end{proposition}
Unlike Proposition \ref{prop.simpfrei} this result is \emph{not} characteristic in that not every set satisfying the conclusion has doubling strictly less than $2$.  Indeed, the doubling of such sets may be much larger than $2$.  In fact a more precise characterisation of sets with doubling in this range is available in the form of Kneser's theorem (\cite{kne::} or \cite[Theorem 5.5]{taovu::}) from which Proposition \ref{prop.simkneser} follows via something called a covering argument of a sort we shall see later in \S\ref{sec.covering}.

The example to highlight the limitation of Proposition \ref{prop.simpfrei} was the first in a series of examples generated by longer and longer arithmetic progressions and these examples go some way to explaining why $2$ should be a critical point in Proposition \ref{prop.simkneser}.  Indeed, if $G=\Z$ and $A$ is a finite arithmetic progression then $A$ has doubling $(2-1/|A|)$.  If some version of the conclusion of Proposition \ref{prop.simkneser} were to hold without the dependence on $\varepsilon$ then we should need to cover $A$ by $O(1)$ cosets of a subgroup $H \leq G$ of size $O(|A|)$.  Of course the only finite subgroup of $\Z$ is $\{0\}$ and so this is not possible.  

This last example shows us that if we are to have a hope of describing sets with doubling $2$ then we shall need to admit another form of structure: long arithmetic progressions.  An arithmetic progression can be thought of as a discrete representation of an interval and in this light can be seen as a special case of a more general structure which, it turns out, also has small doubling: lattices in convex bodies.

A \textbf{centred convex progression} is a set $P$ in $G$, a symmetric convex body $Q$ in $\R^d$ and a homomorphism $\phi:\Z^d \rightarrow G$ such that $\phi(\Z^d\cap Q) = P$.  We say that $P$ is \textbf{$d$-dimensional} and we shall usually simply talk about the set $P$ with $Q$ and $\phi$ being implied (despite the fact that they are not necessarily well-defined).

Given this definition a (symmetric) arithmetic progression is a $1$-dimensional centred convex progression and all $1$-dimensional centred convex progressions are (symmetric) arithmetic progressions.

A convex body in $\R^d$ has doubling $2^d$, and it turns out that this doubling property is inherited by $d$-dimensional convex progressions in the sense that they have doubling $\exp(O(d))$.  The proof of this is not very difficult and can be done using a covering argument.  The details are in Lemma \ref{lem.ccppoly} to avoid breaking the flow.

Given a set of small doubling we can always create a new set with small doubling by adding a subgroup.  In light of this we define a \textbf{$d$-dimensional centred convex coset progression} to be a set of the form $P+H$ where $P$ is a $d$-dimensional centred convex progression and $H$ is a subgroup of $G$; this also has doubling $\exp(O(d))$.  (Again, see Lemma \ref{lem.ccppoly} for a proof.)

With this new type of structure we can set about constructing a large class of sets with small doubling (small here meaning $O(1)$).  In our earlier discussion we found two methods of producing sets with small doubling from subgroups: we could take large subsets and we could take a union of a small number of cosets.  We now replace `subgroup' in these constructions by `centred coset progression'.

Suppose that $A$ is $\exp(d)$-covered by a $d$-dimensional centred convex coset progression $M$ of size at most $\exp(d)|A|$.  Then by definition there is a set $X$ of size at most $\exp(d)$ such that $A \subset X+M$ whence
\begin{equation}\label{eqn.calc}
|A+A| \leq |X+M+X+M| \leq |X|^2|M+M| =\exp(O(d))|A|,
\end{equation}
so that $A$ has doubling $\exp(O(d))$.  Remarkably it turns out that the above is the \emph{only} way of constructing sets of small doubling.
\begin{theorem}[Green-Ruzsa theorem; Fre{\u\i}man's theorem for Abelian groups]\label{thm.gr}
Suppose that $|A+A| \leq K|A|$.  Then $A$ is $\exp(d(K))$-covered by a $d(K)$-dimensional centred convex coset progression $M$ of size at most $\exp(d(K))|A|$.
\end{theorem}
The result above was first proved by Fre{\u\i}man \cite{fre::} for the case of $G$ torsion-free and later a new proof with better bounds was given (for the same setting) by Ruzsa in \cite{ruz::9}.  In \cite{ruz::01} Ruzsa proved the result for groups of bounded exponent which is in some sense at the other end of the spectrum from torsion-free, and then Green and Ruzsa in \cite{greruz::0} established the result above for arbitrary (Abelian) groups with another proof appearing a little later in \cite[Theorem 5.43]{taovu::}.

While Theorem \ref{thm.gr} resolves the qualitative question of the structure of sets with small doubling, the quantitative question remains and this is where most of the recent advances have been.  In their first proof of Theorem \ref{thm.gr} Green and Ruzsa showed that one may take
\begin{equation*}
d(K)=O(K^{4+o(1)}).
\end{equation*}
Various strengthenings were available at that time for torsion-free and groups of bounded exponent. (See, for example, \cite{cha::0} or the appendix to \cite{bou::1} for the torsion-free case, and \cite{gretao::3} for the bounded exponent case.)  Unfortunately, all bounds were of the form $d(K)=O(K^C)$ for some $C>0$, and it was seen as a significant open problem to show $d(K)=O(K^{o(1)})$.

In \cite{sch::1} Schoen made a striking breakthrough proving a bound of the form\footnote{This is our first use of $\log$s in this survey, and they will appear a lot more.  We shall always think of the argument as being larger than some constant, but if the reader does not wish to concern themselves with this then they may think of $\log x$ as denoting $\log (2+x)$.}
\begin{equation*}
d(K)=O(\exp(O(\sqrt{\log K}))),
\end{equation*}
and then shortly after that Croot and Sisask came out with an important new argument in \cite{crosis::} which it turned out could be used to prove
\begin{equation*}
d(K)=O(\log^{3+o(1)}K).
\end{equation*}
Establishing this is one of the main goals of this survey; to be clear we shall prove the following version of Theorem \ref{thm.gr}.
\begin{theorem}[Green-Ruzsa theorem, good bounds]\label{thm.finalfre}
Suppose that $|A+A| \leq K|A|$.  Then $A$ is $\exp(O(\log^{3+o(1)}K))$-covered by an $O(\log^{3+o(1)}K)$-dimensional centred convex coset progression $M$ of size at most $\exp(O(\log^{3+o(1)}K))|A|$.
\end{theorem}
This result with a power of $6$ instead of $3$ was shown in \cite{san::00}, and in the basic framework of this paper that $6$ improves to a $4$.  An improvement of the $4$ to a $3$ is the result of a wonderful iterative application of our basic tool which is due to Konyagin.

For comparison the calculation in the construction before Theorem \ref{thm.gr} turns out to be tight and it follows from this that $d(K)=\Omega(\log K)$, and this is conjecturally the correct order of magnitude.  To see this suppose that $|A+A| = K|A|$ and note by the calculation in (\ref{eqn.calc}) that
\begin{equation*}
K|A| \leq \exp(2d(K))\exp(O(d(K)))\exp(d(K))|A|=\exp(O(d(K)))|A|,
\end{equation*}
from which the lower bound on $d(K)$ follows.
\begin{conjecture}[Polynomial Fre{\u\i}man-Ruzsa conjecture]
Suppose that $A$ has $|A+A| \leq K|A|$.  Then $A$ is $\exp(O(\log K))$-covered by an $O(\log K)$-dimensional centred convex coset progression $M$ of size at most $\exp(O(\log K))|A|$.
\end{conjecture}
We have skipped over a number of the details in this introduction, but before moving on to a more careful discussion it is worth making a couple of remarks on why Fre{\u\i}man's theorem is important.

First there is a practical reason: as a result of the celebrated work of Gowers \cite{gow::4,gow::0} in the late 90s Fre{\u\i}man's theorem has found a bevy of applications.  For example, Gowers himself used it to spectacularly improve the bounds in Szemer{\'e}di's theorem; Szemer{\'e}di and Vu used it to investigate long arithmetic progressions in \cite{szevu::0}; Tao and Vu used it to investigate random matrices in \cite{taovu::0}; Schoen records many shorter consequences at the end of his paper \cite{sch::1} on Fre{\u\i}man's theorem; and Chang in \cite{cha::5} collects together a number of other applications where good bounds would be particularly useful.  There is some discussion of applications at the end of the paper in \S\ref{sec.apps}.

Secondly there are good theoretical reasons, three of which we shall record now.  They may not all make precise sense at this point in the article, but part of our hope is that we shall be able to go some way towards explaining them.
\begin{enumerate}
\item The hypothesis of the theorem is easily satisfied.  In a sense we have seen that this is true empirically as a result of the many applications.  From a theoretical perspective this is because convex coset progressions are ubiquitous in contrast to subgroups (in some groups).  An example to bear in mind is $G=\Z/p\Z$ for $p$ a prime.  This has a very poor subgroup structure, but since arithmetic progressions are convex coset progressions we see immediately that there is an abundance of convex progressions.
\item A convex coset progression supports a lot of structure. While it is not a coset, it behaves enough like a coset that it can support many commonly used analytic arguments, and in particular a sort of approximate harmonic analysis.  This means that many results for groups can also be established for convex coset progressions.  The pioneering work here is that of Bourgain \cite{bou::5} which was framed in a level of generality which includes convex coset progressions by Green and the author in \cite{gresan::0}.
\item Finally, the result is a rough equivalence: any set satisfying the conclusion of the theorem satisfies the hypothesis with $K$ replaced by $\exp(O(d(K)))$.  Thus the better the bound on the function $d(K)$ the less loss there is in passing from the implicit algebraic data that a set has small doubling to the explicit algebraic data that it is generated from a convex coset progression.
\end{enumerate}
The paper now splits as follows.  In the next section, \S\ref{sec.overview}, we describe the main plan of attack on Fre{\u\i}man's theorem which roughly splits it into two parts.  The first part is covered in \S\S\ref{sec.covering}--\ref{sec.kony}; the second in \S\S\ref{sec.rpgccp}--\ref{sec.lcp}.  There is a concluding section in \S\ref{sec.con}, and also a section on Pl{\"u}nnecke's inequality in \S\ref{sec.plun} which is a basic tool in the structure theory of set addition and has recently received a fantastic new proof by Petridis.

\section{Overview}\label{sec.overview}

The proof of Theorem \ref{thm.finalfre} splits naturally into two parts: one covers the more combinatorial aspects, and one the more harmonic analytic aspects.  This particular de-coupling can be said to originate with the work of Green and Ruzsa \cite{greruz::0}, although their focus was much more on the second of the two, while the more recent improvements to the bounds have arisen (largely) from more careful combinatorial analysis in the first part of the argument.

The key definition is that of relative polynomial growth: to be clear we say that a set $X$ has \textbf{relative polynomial growth of order $d$} if
\begin{equation*}
|nX| \leq n^d|X| \text{ for all } n \in \N.
\end{equation*}
One might reasonably wish to insert a constant in front of the term on the right hand side, but we shall find that we are easily able to absorb this into the dimension at little cost to the quality of our eventual bounds.

It is worth noting that having relative polynomial growth is \emph{a priori} stronger than a small doubling condition.  It will turn out later (see Proposition \ref{prop.usefulfirst}) that the conditions are qualitatively equivalent in that doubling $K$ implies relative polynomial growth of order $O_K(1)$, but quantitatively this equivalence entails an exponential loss and is the reason for the exponential weakness of the original arguments of Green and Ruzsa.

With the definition above the argument splits into the following two parts.
\begin{enumerate}
\item \emph{(From small doubling to relative polynomial growth)} Given a set $A$ with $|A+A| \leq K|A|$ we find a \textbf{symmetric neighbourhood of the identity}, $X$, (meaning that $X=-X$ and $0_G \in X$) of size at most $O_K(|A|)$ with relative polynomial growth of order $O_K(1)$ such that $A$ is $O_K(1)$-covered by $X$.
\item \emph{(From relative polynomial growth to convex coset progressions)} Given a symmetric neighbourhood of the identity, $X$, with relative polynomial growth of order $d$ we show that $X$ is contained in an $O_d(1)$-dimensional centred convex coset progression of size at most $O_d(|X|)$.
\end{enumerate}
Note that if we had proved these two statements then they combine to give Theorem \ref{thm.gr}.  We now turn to look at these two parts in a little more detail.

\subsection{From small doubling to relative polynomial growth}
The starting point here are the covering arguments of Ruzsa which will be developed in \S\ref{sec.covering}, and which will be related to relative polynomial growth in \S\ref{sec.rpg}.  As we shall see there it is possible to use these covering arguments to show that if $|A+A| \leq K|A|$ then $A$ has relative polynomial growth of order $O(K^4)$ and from there it is a short step to the following corollary.
\begin{corollary}\label{cor.ff}
Suppose that $|A+A| \leq K|A|$.  Then $A$ is $1$-covered by a symmetric neighbourhood of the identity of size at most $\exp(O(\log K ))|A|$ and relative polynomial growth of order $O(K^{4})$.
\end{corollary}
This result is much weaker than we should like, but it turns out that it is essentially so because it provides a set which $1$-covers.  In \S\ref{sec.con} we discuss an example of a set $A$ with doubling $K$ such that any set $1$-covering it must have either relative polynomial growth of order $\Omega(K)$ or size $\exp(\Omega(K))|A|$.  Thus to improve the bound on the order of relative polynomial growth we shall need to increase the covering number.

In \S\ref{sec.wb} we discuss a general framework for improving the above Corollary \ref{cor.ff} before \S\ref{sec.cs} where we introduce a key new tool: the Croot-Sisask lemma.  \S\ref{sec.cs} includes the following result which can be seen as representing the state of the art prior to Schoen \cite{sch::1} and Croot and Sisask \cite{crosis::} (although we shall use a special case of the Croot-Sisask lemma to prove it).
\begin{proposition}\label{prop.oldstateofart}
Suppose that $|A+A| \leq K|A|$.  Then $A$ is $\exp(O(K^{1+o(1)}))$-covered by a symmetric neighbourhood of the identity of size at most $\exp(O(\log K ))|A|$ and relative polynomial growth of order $O(K^{1+o(1)})$.
\end{proposition}
In \S\ref{sec.kony} we shall then make much more effective use of the Croot-Sisask lemma to show the following.
\begin{proposition}\label{prop.grow1}
Suppose that $|A+A| \leq K|A|$.  Then $A$ is $\exp(O(\log^{4}K))$-covered by a symmetric neighbourhood of the identity of size at most $\exp(O(\log K))|A|$ and relative polynomial growth of order $O(\log^{4}K)$.
\end{proposition}
This result is where most of the more recent new material appears, but there is then also a combinatorial refinement following Konyagin which leads to our strongest result at the end of \S\ref{sec.kony}.
\begin{proposition}\label{prop.grow2}
Suppose that $|A+A| \leq K|A|$.  Then $A$ is $\exp(O(\log^{3+o(1)}K))$-covered by a symmetric neighbourhood of the identity  of size at most $\exp(O(\log^{1+o(1)}K))|A|$ and relative polynomial growth of order $O(\log^{3+o(1)}K)$.
\end{proposition}

\subsection{From relative polynomial growth to convex coset progressions}\label{sec.rpgccpintro}
To pass from relative polynomial growth to convex coset progressions it is useful to start by considering some examples of sets with relative polynomial growth.  Of course, if $Q$ is a convex body in $\R^d$ then $\mu(nQ)=n^d\mu(Q)$ for all $n \geq 1$ and so one expects that any $d$-dimensional centred convex coset progression has relative polynomial growth roughly $d$ (in fact $d^{1+o(1)}$).

Now, if $P$ is a $d$-dimensional centred convex coset progression and $X \subset P$ has size $\exp(-d^{1+o(1)})|P|$ then
\begin{equation*}
|nX| \leq |nP| \leq n^{d^{1+o(1)}}|P| = O(n)^{d^{1+o(1)}}|X| = n^{d^{1+o(1)}}|X| \text{ for all } n \geq 1.
\end{equation*}
Crucially, though, a union of $\exp(d^{1+o(1)})$ translates of centred convex coset progressions will (generically) have relative polynomial growth of order $\exp(d^{1+o(1)})$ and \emph{not} $d^{1+o(1)}$ so that relative polynomial growth distinguishes between covering and containment in a way that doubling does not.

It turns out that there is a matching result which tells us that essentially the only way of creating sets of relative polynomial growth is by the above method.
\begin{theorem}\label{thm.rtc}
Suppose that $X$ has relative polynomial growth of order $d$.  Then there is a centred convex coset progression $M$ such that
\begin{equation*}
X-X \subset M, \dim M =O(d\log^2 d) \text{ and } |M| \leq \exp(O(d\log^2 d))|X|.
\end{equation*}
\end{theorem}
The first thing to say is that the dimension here is tight up to factors of $\log d$.  This can be seen by, for example, letting $X$ be the cube of side length $N$ in $\Z^d$.  This has polynomial growth of order $\Omega(d)$ and any convex coset progression containing $X$ has tripling at least $2^d$ by the discrete Brunn-Minkowski inequality (see, \emph{e.g.} \cite[Lemma 2.4]{gretao::8}) and so has dimension $\Omega(d)$.

This result is the part of the argument which uses harmonic analysis and itself splits into a number of parts.  These are covered in the second part of the paper starting at \S\ref{sec.rpgccp}.

We should remark that Theorem \ref{thm.finalfre} follows immediately from Proposition \ref{prop.grow2} and Theorem \ref{thm.rtc}.

\section{Pl{\"u}nnecke's inequality}\label{sec.plun}

This section is the final section before we plunge into the proof of Fre{\u\i}man's theorem and it will cover the invaluable tool of Pl{\"u}nnecke's inequality following the exciting new work by Petridis \cite{pet::0,pet::}.  The discussion in his papers is more comprehensive than ours and we direct the reader interested in more details there, but we hope to cover the salient features in what follows.

Our starting point is the observation that \emph{given} Fre{\u\i}man's theorem if $|A+A| \leq K|A|$ then there is an $O_K(1)$-dimensional centred convex coset progression $M$ of size $O_K(|A|)$ such that $A$ is $O_K(1)$-covered by $M$.  This means that there is some set $X$ of size $O_K(1)$ such that $A \subset X+M$.  On the other hand as remarked in \S\ref{sec.rpgccpintro} the set $M$ has relative polynomial growth of order $O_K(1)$, hence
\begin{equation*}
|nA| \leq |nX||nM| \leq |X|^{n}(n)^{O_K(1)}|M| = O_K(1)^{n}|A| \text{ for all } n \in \N.
\end{equation*}
It turns out that a much stronger inequality is true:
\begin{theorem}[Pl{\"u}nnecke's inequality]\label{thm.plun}
Suppose that $|A+A| \leq K|A|$.  Then 
\begin{equation*}
|nA| \leq K^{n}|A|\text{ for all }n \in \N.
\end{equation*}
\end{theorem}
This result is due to Pl{\"u}nnecke's \cite{plu::} and was rediscovered and greatly developed by Ruzsa \cite{ruz::5}.  Both Ruzsa and Pl{\"u}nnecke's arguments were graph theoretic and quite involved appealing to Menger's theorem (see \cite[\S6.5]{taovu::} for details).  In \cite{pet::0} Petridis removed the need for Menger's theorem and then a little later in \cite{pet::} he found a wonderful entirely new proof.

The core of Petridis' argument is the next lemma.  The idea is that if we are given sets $A$ and $X$ such that $|A+X| \leq K|X|$ then it is a good idea to pass to the `best' possible subset of $X$.  That is to say, to pass to the subset $X'$ of $X$ for which $|A+X'|/|X'|$ is minimal.  Turning this around if $X$ is already the `best' subset then every $X' \subset X$ has $|A+X'|/|X'|$ bigger than $|A+X|/|X|$.  In this case Petridis proved the following beautiful lemma.
\begin{lemma}\label{lem.petridis} Suppose that $|A+X| \leq K|X|$ and $|A+X'| \geq K|X'|$ for all $X' \subset X$.  Then for all (finite) sets $C$ we have
\begin{equation*}
|A+X+C| \leq K|X+C|.
\end{equation*}
\end{lemma}
\begin{proof}
We iteratively decompose $X+C$ into disjoint sets contained in translates of $X$: $X+C=\sqcup_c{X_c}$ where $X_c \subset X+c$.  Writing $Y_c:=(X+c) \setminus X_c$ we have
\begin{eqnarray*}
|A+X+C| \leq  \sum_c{|A+X_c|} &=&\sum_{c}{|A+((X+c) \setminus Y)|}\\ &\leq & \sum_c{(|A+X+c| - |A+Y_c|)}\\ & \leq & \sum_{c}{(K|X+c| - K|Y_c|)} = \sum_c{K|X_c|} = K|X+C|.
\end{eqnarray*}
The result is proved.
\end{proof}
Given the idea of passing to this `best' possible $X$ the proof is rather natural, but the reader should make no mistake: the idea to do this is very nice and eluded many people!

Petridis then gives the following immediate corollary.
\begin{corollary}\label{cor.petcor}
Suppose that $|A+B| \leq K|B|$.  Then there is some non-empty $X \subset B$ such that
\begin{equation*}
|nA+X| \leq K^n|X| \text{ for all }n \in \N.
\end{equation*}
\end{corollary}
\begin{proof}
We can pick $X \subset B$ such that $|A+X|/|X|$ is minimal over (non-empty) subsets of $B$.  In this case $A$ and $X$ satisfy the hypotheses of Petridis' lemma and hence the conclusion.  Applying the conclusion with $C=(n-1)A$ we get that
\begin{equation*}
|X+ nA| \leq K |X+(n-1)A| \text{ for all } n \in \N,
\end{equation*}
and this gives the result (by induction).
\end{proof}
Note that Pl{\"u}nnecke's inequality (Theorem \ref{thm.plun}) follows immediately from this applied to the set $A$ and $B=A$ since $X \subset B=A$.

It is also possible to control jointly positive and negative sums of $A$ using the following result called Ruzsa's triangle inequality \cite{ruz::02} (see also \cite[Lemma 2.6]{taovu::}).
\begin{lemma}[Ruzsa's triangle inequality]  Suppose that $|A-B| \leq K|B|$ and $|B-C| \leq L|B|$.  Then
\begin{equation*}
|A+C| \leq KL|B|.
\end{equation*}
\end{lemma}
\begin{proof}
We consider the map $B \times (A+C) \rightarrow (A-B) \times (B-C)$ defined by $(b,s) \mapsto (a(s)-b, b-c(s))$ where $a(s)$ and $c(s)$ are functions on $A+C$ such that $a(s) \in A$, $c(s) \in C$ and $a(s)+c(s)=s$.  It is easy to check that our map on $B \times (A+C)$ is an injection: suppose that
\begin{equation*}
(a(s)-b, b-c(s)) = (a(s')-b',b'-c(s')),
\end{equation*}
then adding we get that $s=a(s)+c(s) = a(s')+c(s') = s'$ and so $s=s'$, and hence $b=b'$.  It follows from this that $|B||A+C| \leq |A-B||B-C|$ and we have the result.
\end{proof}
It may be intuitively helpful to know that this can be seen as the triangle inequality for a certain pseudo-metric one can define on sets (in groups) called the Ruzsa distance.  (See \cite[\S2.3]{taovu::} for more details.)

As an immediate corollary of our work so far we have the so-called Pl{\"u}nnecke-Ruzsa inequalities which are slightly more general than Pl{\"u}nnecke's inequality.
\begin{corollary}[Pl{\"u}nnecke-Ruzsa inequalities]
Suppose that $|A+A| \leq K|A|$.  Then 
\begin{equation*}
|nA-mA| \leq K^{n+m}|A|\text{ for all }n,m \in \N.
\end{equation*}
\end{corollary}
\begin{proof}
Apply Corollary \ref{cor.petcor} to get a set $X\subset A$ such that $|rA +X| \leq K^r|X|$ for all $r \in \N$ so that in particular $|nA +X| \leq K^n|X|$ and $|-mA-X| \leq K^m|X|$.  It follows from Ruzsa's triangle inequality that $|nA-mA| \leq K^{n+m}|X| \leq  K^{n+m}|A|$ as required. 
\end{proof}

\section{Ruzsa's covering lemma}\label{sec.covering}

Pl{\"u}nnecke's inequality showed us how small doubling leads to small higher order sums.  In \cite{ruz::01} Ruzsa introduced another argument called a covering argument to the area which yields quantitatively similar order results to Pl{\"u}nnecke's inequality but has the advantage of also providing a little structure.  This covering argument is the topic of this section and will already give us a version of Fre{\u\i}man's theorem in groups of bounded exponent.  We start with the basic lemma:
\begin{lemma}[Ruzsa's covering lemma, {\cite[Lemma 2.14]{taovu::}}]\label{lem.rcl}
Suppose that $|A + S| \leq K|S|$. Then there is a set $T \subset A$ with $|T| \leq K$ such that $A \subset T+ S-S$.
\end{lemma}
\begin{proof}
The technique here is very powerful so it is worth developing in some detail: we let $T \subset A$ be maximal $S$-separated.  (The set $T$ is \emph{$S$-separated} if every pair of distinct elements $t,t' \in T$ have $t+S$ and $t'+S$ disjoint.)  It follows that $|T+S| = |T||S|$.  On the other hand, since $T \subset A$, we have $T+S \subset A+S$ and so
\begin{equation*}
|T||S| = |T+S| \leq |A+S| \leq K|S|;
\end{equation*}
we conclude that $|T| \leq K$.

Now we use the fact that $T$ is maximal: if $a \in A$ then (either trivially if $a \in T$ or) by maximality there is some $t \in T$ such that $(t+S) \cap (a+S) \neq \emptyset$.  It follows that $a \in t+S-S \subset T+S-S$ and the result is proved.
\end{proof}
It should be remarked that this has an extension developed by Tao in \cite{tao::6} giving a non-Abelian version of a (slightly weak) Pl{\"u}nnecke inequality, although now Petridis' approach to Pl{\"u}nnecke's inequality also yields a non-Abelian version of the (almost) full strength Pl{\"u}nnecke inequality.

Lemma \ref{lem.rcl} (or rather the technique used to prove it) can also be used to show that $d$-dimensional centred convex progressions have doubling $\exp(O(d))$.
\begin{lemma}\label{lem.ccppoly}
Suppose that $M$ is a $d$-dimensional centred convex coset progression.  Then $|M+M| \leq\exp(O(d))|M|$.
\end{lemma}
\begin{proof}
To start with we write $M=P+H$ for some centred convex progression $P$ and $Q \subset \R^d$ for the convex body generating $P$.  Given $\lambda \in \R_{>0}$ we write $\lambda Q$ for the set $Q$ dilated by a factor $\lambda$ so that $\mu(\lambda Q) = \lambda^d\mu(Q)$.

Now, let $X \subset 2Q$ be a maximal $\frac{1}{4}Q$-separated set so that by the same argument as in Ruzsa's covering lemma we have
\begin{equation*}
2Q \subset X + \frac{1}{4}Q - \frac{1}{4}Q = X+\frac{1}{2}Q \text{ and } |X| 4^{-d}\mu(Q) \leq (9/4)^d\mu(Q).
\end{equation*}
From the second of these it follows that $|X| \leq 9^d$.  With the first we note that
\begin{equation*}
P+P=\phi(Q \cap \Z^d)+\phi(Q \cap \Z^d) \subset \phi(2Q \cap \Z^d) \subset \bigcup_{x \in X}{\phi((x+\frac{1}{2}Q)\cap \Z^d)}.
\end{equation*}
Let $T$ be a set such that if $\phi((x+\frac{1}{2}Q )\cap \Z^d)\neq \emptyset$ for some $x \in X$, then $T$ contains exactly one element of this set, so that $|T| \leq |X|$.  Then if $t' \in \phi((x+\frac{1}{2}Q )\cap \Z^d)$ we have some $t \in T$ such that $t-t' \in \phi(Q \cap \Z^d)=P$, whence $P+P \subset T+P$.  Adding $H$ we get that $(P+H) + (P+H) \subset T+ (P+H)$ since $H+H=H$, and the result follows given the bound on $|X|$ (and hence $|T|$).
\end{proof}

One informative illustration of why Ruzsa's covering lemma is so powerful is given in Ruzsa's original paper \cite{ruz::01}.
\begin{proposition}[Fre{\u\i}man-Ruzsa theorem for groups of bounded exponent]\label{prop.ruzanal}
Suppose that $G$ has exponent $r$ and $|A+A| \leq K|A|$.  Then $\langle A \rangle$, the group generated by $A$, has size at most $K^2r^{K^4}|A|$.
\end{proposition}
\begin{proof}
The idea is simply to apply Ruzsa's covering lemma to $2A-A$.  By the Pl{\"u}nnecke-Ruzsa inequalities we have that $|(2A-A) + A | =|3A-A| \leq K^4|A|$ and so there is a set $T$ of size at most $K^4$ such that
\begin{equation*}
A+(A-A) = 2A-A \subset T+A-A.
\end{equation*}
By induction it follows that $nA+ (A-A) \subset nT + (A-A)$ for all $n \in \N$.  We write $H$ for the group generated by $T$ and note that $|H| \leq r^{|T|}$ and $nT + A-A \subset H+A-A$.  We conclude that $nA\subset nA +(A-A) \subset H+A-A$ for all $n$ and similarly for $-nA$ since $H$ and $A-A$ are symmetric.  It follows that $\langle A \rangle \subset H+A-A$ and we get the result since $|A-A| \leq K^2|A|$.
\end{proof}
Ruzsa has a further argument published in \cite{deshenpla::} which improves the $K^4$ above to a $K^3$ using a slight refinement of the Pl{\"u}nnecke-Ruzsa inequalities.  A refined covering argument of Green and Ruzsa \cite{greruz::} gives the best known result following from covering techniques, while the best known upper bound by any argument is due to Schoen \cite{sch::1} who showed that the group generated by $A$ has size at most $r^{K^{1+o(1)}}|A|$.

It may also be worth noting that by letting $A$ be $2K+1$ independent elements the upper bound is at least $r^{2K+1}|A|$ so that Schoen's result is tight up to the $o(1)$-term.  (In the case when $r=2$ this $o(1)$-term has been eliminated via some arguments from extremal set theory introduced by Green and Tao.  We shall not pursue this here but see \cite{gretao::2,kon::0} and \cite{zoh::} for details.)

\section{Relative polynomial growth}\label{sec.rpg}

In \S\ref{sec.overview} we made it clear that relative polynomial growth was going to be a key concept for us, and it arises naturally when we compare the results of \S\ref{sec.plun} with those of \S\ref{sec.covering} as we shall now see.

Suppose that $|A+A| \leq K|A|$.  By Pl{\"u}nnecke's inequality we have that
\begin{equation*}
|nA| \leq K^n|A| \text{ for all }n \in \N.
\end{equation*}
On the other hand, by an inductive application of Ruzsa's covering lemma (as in the proof of Proposition \ref{prop.ruzanal}) we have that $(n-1)A + (A-A) \subset (n-1)T + A-A$ for all $n \in \N$ and some $T$ of size at most $K^4$.  Now since $G$ is Abelian we have
\begin{equation*}
|(n-1)T| \leq \binom{|T|+n-2}{|T|-1}\leq n^{|T|},
\end{equation*}
and so
\begin{equation*}
|nA| \leq |(n-1)T|.|A-A| \leq K^2|(n-1)T||A| \leq n^{O(K^4)}|A|
\end{equation*}
for all $n\in \N$.  For small values of $n$ this is much weaker than Pl{\"u}nnecke's inequality but for large values of $n$, the estimate from Pl{\"u}nnecke is exponential while this is polynomial.

Proposition \ref{prop.ruzanal} does not adapt directly to the case of general Abelian groups because when $G$ does not have bounded exponent we cannot expect the group generated by $A$ to be finite (consider, for example, $G=\Z$), but as we saw above it is sufficient to give relative polynomial growth.
\begin{proposition}\label{prop.usefulfirst}
Suppose that $A \subset G$ has $|A+A| \leq K|A|$.  Then $A$ has relative polynomial growth of order $O(K^4)$.
\end{proposition}
We have an immediate corollary of this in the following.
\begin{corollary*}[Corollary \ref{cor.ff}]
Suppose that $|A+A| \leq K|A|$.  Then $A$ is $1$-covered by a symmetric neighbourhood of the identity of size at most $\exp(O(\log K ))|A|$ and relative polynomial growth of order $O(K^{4})$.
\end{corollary*}
\begin{proof}
Since $A$ has relative polynomial growth of order $O(K^4)$ by Proposition \ref{prop.usefulfirst} we see by Ruzsa's triangle inequality that $|n(A-A)| \leq |nA - A||-A-nA|/|A| = n^{O(K^4)}|A|$ and so $A-A$ also has relative polynomial growth of order $O(K^4)$.  On the other hand $|A-A| \leq K^2|A|$ and $A-A$ is a symmetric neighbourhood of the identity which $1$-covers $A$ and so we are done.
\end{proof}
The weakness of this result is that it is exponentially expensive to apply: if $A$ has relative polynomial growth of order $d$ then $A$ trivially has doubling at most $2^d$.  This means that if $A$ has doubling $K$ and we apply the proposition we get that $A$ has polynomial growth of order $O(K^4)$, but then we conclude that $A$ has doubling at most $\exp(O(K^4))$ -- an exponential loss.  Incidentally, this exponential loss is exactly the reason for the exponential loss in Green and Ruzsa's first version of Fre{\u\i}man's theorem.

To deal with this situation we have a slight refinement of Ruzsa's covering lemma due to Chang \cite{cha::0}.  Chang observed that if a set has a sort of relative sub-exponential growth on one scale then the covering set $T$ in Ruzsa'c overing lemma can be made to be highly structured and hence get relative polynomial growth of much lower order on all scales.  To be clear we need some notation: write
\begin{equation*}
\Span(X):=\{\sigma.X:=\sum_{x \in X}{\sigma_xx}: \sigma \in \{-1,0,1\}^X\}.
\end{equation*}
Then we have the following result.
\begin{lemma}[A variant of Chang's covering lemma]\label{lem.ccl}
Suppose that $|kA + S| < 2^k|S|$ (and $0_G \in A$). Then there is a set $T \subset A$ with $|T| < k$ such that $A \subset \Span(T) + S-S$.
\end{lemma}
\begin{proof}
Let $T$ be a maximal \emph{$S$-dissociated} subset of $A$, that is a maximal subset of $A$ such that
\begin{equation*}
(\sigma.T + S) \cap (\sigma'.T+S) = \emptyset \text{ for all } \sigma\neq \sigma' \in \{0,1\}^T.
\end{equation*}
Now suppose that $x' \in A \setminus T$ and write $T':=T \cup \{x'\}$. By maximality of $T$ there are elements $\sigma,\sigma' \in \{0,1\}^{T'}$ such that $(\sigma.T' +S) \cap (\sigma'.T' + S) \neq \emptyset$. Now if $\sigma_{x'} = \sigma'_{x'}$ then $(\sigma|_T.T +S) \cap (\sigma'|_T.T + S) \neq \emptyset$ contradicting the fact that $T$ is $S$-dissociated. Hence, without loss of generality, $\sigma_{x'}=1$ and $\sigma'_{x'}=0$, whence
\begin{equation*}
x' \in \sigma'|_T.T -\sigma|_T.T + S - S \subset \Span(T) + S-S.
\end{equation*}
We are done unless $|T|\geq k$; assume it is and let $T' \subset T$ be a set of size $k$. Denote $\{\sigma.T': \sigma \in \{0,1\}^{T'}\}$ by $P$ and note that $P \subset kA$ (since $0_G \in A$), whence
\begin{equation*}
2^k|S| = |P+S| \leq |kA+S| < 2^k|S|.
\end{equation*}
This contradiction completes the proof.
\end{proof}
Dissociativity is a very important concept in harmonic analysis and the relative version introduced in the above proof also have many uses.  The reader interested in learning more is directed to \cite[\S4.5]{taovu::} or the book \cite{rud::1} or Rudin.

This result yields the following useful corollary.
\begin{corollary}\label{cor.useful}
Suppose that $X \subset G$ is a symmetric neighbourhood and $|(3k+1)X| < 2^k|X|$ for some $k \in \N$.  Then $X$ has relative polynomial growth of order $O(k)$.
\end{corollary}
\begin{proof}
Apply Lemma \ref{lem.ccl} to the sets $3X$ and $X$ to get a set $T$ of size less than $k$ such that $3X \subset \Span(T) +2X$.  It follows that $nX \subset (n-2)\Span(T) + 2X$ and so
\begin{equation*}
|nX| \leq |(n-2)\Span(T)||2X| \leq (2n-3)^k|2X|=O(n)^k|X|
\end{equation*}
provided $n \geq 2$.  We conclude that $X$ has relative polynomial growth $O(k)$ as required.
\end{proof}
This will often be combined with the following useful application of Ruzsa's covering lemma.
\begin{lemma}\label{lem.convert}
Suppose that $X$ is a set of relative polynomial growth of order $d$ and $|A+X| \leq K|X|$.  Then $A$ is $K$-covered by $X-X$, a symmetric neighbourhood of the identity having relative polynomial growth of order $O(d)$.
\end{lemma}
\begin{proof}
We just apply Ruzsa's covering lemma to get that $A$ is $K$-covered by $X-X$.  This is a symmetric neighbourhood of the identity and $|n(X-X)| \leq n^{O(d)}|X|$ by Ruzsa's triangle inequality and the fact that $X$ has relative polynomial growth of order $d$.  The result follows.
\end{proof}

\section{Bogolyubov-Ruzsa-type lemmas}\label{sec.wb}

In the last section we proved Proposition \ref{prop.usefulfirst} which converted our small doubling condition into a relative polynomial growth of low order condition.  As mentioned there this was not a particularly efficient process and so we set about proving Corollary \ref{cor.useful} to do better.  In this section we shall discuss a general framework for using this corollary.

To start with suppose that $X$ (is symmetric and) has $|X+X| \leq K|X|$.  By Pl{\"u}nnecke's inequality we have that
\begin{equation*}
|(3k+1)X| \leq K^{3k+1}|X| \text{ for all }k \geq 1
\end{equation*}
which is \emph{not} smaller than $2^k$ (unless $K$ is very small which is a case we have already discussed in the introduction).  To get a sub-exponential estimate it will be useful to have a result of the following shape.
\begin{proposition}[Weak Bogolyubov-Ruzsa-type lemma]\label{prop.genbog}
Suppose that $X$ is symmetric with $|X+X| \leq K|X|$ and $m \in \N$.  Then there is a symmetric neighbourhood of the identity, $T$, such that
\begin{equation*}
|T| =\Omega_{m,K}(|X|) \text{ and } mT \subset 4X.
\end{equation*}
\end{proposition}
Before remarking on the proof or history, we should see how such a result can be used to give a set with relative sub-exponential growth.  Given $X$ (symmetric) with $|X+X| \leq K|X|$ we apply the lemma with some parameter $m$ to get a set $T$ as described.  On the other hand by Pl{\"u}nnecke's inequality with parameter $l$ we have that $|4lX| \leq K^{4l}|X|$ and it follows that
\begin{equation*}
|mlT| \leq K^{4l}|X| = O_{m,K}(K^{4l}|T|)=\exp(O_K(l+O_m(1)))|T|.
\end{equation*}
At this point put $3k+1=ml$ and letting $m \rightarrow \infty$ very slowly with $l$ we get
\begin{equation*}
|(3k+1)T| = \exp(o_K(k))|T|.
\end{equation*}
It follows that for $k$ sufficiently large in terms of $K$ the right hand side can be made to be at most $2^k$ and so Corollary \ref{cor.useful} can be applied to the set $T$.  Whether this turns out to be useful or not depends entirely on the quality of the lower bound in Proposition \ref{prop.genbog} and establishing results of that type with good bounds will be a major part of the remainder of the paper.

Returning to the history, in the case when $X$ is a thick set (meaning $|X|=\Omega(|G|)$) Proposition \ref{prop.genbog} follows from work of Bogolyubov \cite{bog::}.  Ruzsa in \cite{ruz::9} introduced results of this type to Fre{\u\i}man's theorem, and the above Proposition does follow from his work.  The difference here is that both Bogolyubov and Ruzsa prove stronger statements, in particular showing that the set $4A$ contains a low dimensional Bohr set (see \S\ref{sec.rpgccp} for a definition); the set $T$ can then be identified as a $1/m$-dilate of this Bohr set.

The structurally weaker version of the Bogolyubov-Ruzsa lemma which we need here is fortunately rather easier to prove and results in stronger bounds.  Since our objective is one of bounds this works out well.  

\section{The Croot-Sisask lemma}\label{sec.cs}

One of the key recent tools which has made advances in Fre{\u\i}man's theorem possible is called the Croot-Sisask lemma.  This was first proved by Croot and Sisask in \cite{crosis::} and then refined by Croot, {\L}aba and Sisask in \cite{croabasis::}.  The aim of this section is to give a proof of the Croot-Sisask lemma and then immediately give an application to Fre{\u\i}man's theorem.  

Before starting we shall need a little notation.  As we are interested in sumsets it will not come as too much surprise that we should be using the convolution of functions.  First, recall that for $p \in [1,\infty)$ the space $\ell^p(G)$ is the space of functions $f:G \rightarrow \C$ endowed with the norm
\begin{equation*}
\|f\|_{\ell^p(G)}:=\left(\sum_{x \in G}{|f(x)|^p}\right)^{1/p}.
\end{equation*}
For infinity we take the usual convention that
\begin{equation*}
\|f\|_{\ell^\infty (G)}:=\max\{|f(x)|: x \in G\},
\end{equation*}
and apart from $\ell^\infty$ there is one other $\ell^p$ space of particular importance, and that is $\ell^2$.  This is also a Hilbert space with inner product defined by
\begin{equation*}
\langle f,g\rangle = \sum_{x \in G}{f(x)\overline{g(x)}} \text{ for all } f,g \in \ell^2(G).
\end{equation*}
Now, given $f,g \in \ell^1(G)$ we define their \textbf{convolution} to be the function $f \ast g$ determined point-wise by
\begin{equation*}
f \ast g(x):=\sum_{y+z=x}{f(y)g(z)} \text{ for all } x \in G.
\end{equation*}
Given a finite set $A \subset G$ we write $\mu_A$ for the uniform probability mass function supported on $A$.  (If $G$ were locally compact rather than discrete then we should define $\mu_A$ as a measure but we do not need to involve the additional analysis here.)

There are two ways in which convolution is useful.  The first is because it is an average: in particular if $f \in \ell^1(G)$ and $A\subset G$ is finite then $f \ast \mu_A(x)$ is the average value of $f$ on $x-A$.  In general this means that the convolution of two functions is smoother than the constituent functions and hence the convolution is easier to analyse.

Secondly, convolution is useful to us because
\begin{equation*}
A+B:=\supp 1_A \ast 1_B \text{ for all } A,B \subset G,
\end{equation*}
so that we can analyse $A+B$ through the (much easier to understand) function $1_A \ast 1_B$.

In a certain sense convolution comes from integrating the regular representation and it will be useful to have some notation for this: we write
\begin{equation*}
\rho:G \rightarrow \Aut(\ell^2(G));x \mapsto (f \mapsto \rho_x(f):G \rightarrow \C; y \mapsto f(x+y)).
\end{equation*}
To be concrete, with the regular representation in hand we have that
\begin{equation*}
f \ast g(x)=\langle f, \rho_{-x}(\tilde{g})\rangle \text{ for all } x \in G
\end{equation*}
where $\tilde{g}(x)=\overline{g(-x)}$ for all $x \in G$.

With this notation we can describe the idea behind the Croot-Sisask lemma.  Suppose that $S$ is an arithmetic progression and $T$ is a much shorter arithmetic progression with the same common difference so that $|S+T| \approx |S|$.

Now the Croot-Sisask lemma will tell us that for $f \in \ell^p(G)$ the function $f \ast \mu_S$ does not change much when we translate by elements of $T$.  To see this we recall from earlier that $f \ast \mu_S(x)$ is the average value of $f$ on $x-S$.  Then if $t \in T$ we have $x-S+t \approx x-S$ so that the average of $f$ over $x-S+t$ is approximately the same as the average of $f$ over $x-S$.

The full Croot-Sisask lemma is the following much stronger version of this argument replacing arithmetic progressions by any set with small doubling.
\begin{lemma}[Croot-Sisask]\label{lem.cs}  Suppose that $f \in \ell^p(G)$ for some $p \geq 2$, $S,T \subset G$ are such that $|S+T| \leq L|S|$, and $\eta \in (0,1]$ and $p \in [2,\infty)$ are parameters.  Then the set of $x$ such that 
\begin{equation*}
\|\rho_x(f \ast \mu_S) - f \ast \mu_S\|_{\ell^p(G)} \leq \eta \|f\|_{\ell^p(G)}
\end{equation*}
is a symmetric neighbourhood of the identity and has size at least $(2L)^{-O(\eta^{-2}p)}|T|$.
\end{lemma}
The proof proceeds by random sampling: the idea is that since $f \ast \mu_S$ is point-wise the average value of $f$ on translates of $S$, this can be well approximated by the average value of $f$ on a small set of `typical' elements of $S$.  We are then done if we let $X$ be the set of elements of $G$ such that translating these typical elements does not vary them very much.  To make the notion of being well approximated precise we shall need an inequality called the Marcinkiewicz-Zygmund inequality, and for this we require a little more notation.  

Given $p \in [1,\infty)$ and $(X,\mu)$ a measure space we write $L^p(\mu)$ for the space (of equivalence classes of) measurable functions on $X$ endowed with the norm
\begin{equation*}
\|f\|_{L^p(\mu)}:=\left(\int{|f(x)|^pd\mu(x)}\right)^{1/p}.
\end{equation*}
\begin{theorem}[Marcinkiewicz-Zygmund inequality]\label{thm.mzi}  Suppose that $p\in [2,\infty)$ and we are given independent random variables $X_1,\dots,X_n \in L^p(\P)$ with $\E{\sum_i{X_i}}=0$.  Then
\begin{equation*}
\|\sum_{i}{X_i}\|_{L^p(\P)} = O\left(\sqrt{p}\|\sum_{i}{|X_i|^2}\|_{L^{p/2}(\P)}^{1/2}\right).
\end{equation*}
\end{theorem}
Intuitively one might like to think of the $X_i$s are independent variance one, mean zero random variables.  Then the central limit theorem suggests that $\sqrt{n}^{-1}\sum_i{X_i} \sim N(0,1)$ and the $p$th moments of the normal distribution are well-known (and in any case easily computed); we have
\begin{equation*}
\|\sum_{i}{X_i}\|_{L^p(\P)}^p = n^{p/2}\cdot \frac{2^{p/2}\Gamma((p+1)/2)}{\sqrt{\pi}} = O\left(\sqrt{p}\|\sum_{i}{|X_i|^2}\|_{L^{p/2}(\P)}^{1/2}\right).
\end{equation*}
Thus the Marcinkiewicz-Zygmund inequality can be thought of as saying that nothing much worse than this can happen.

There is a special case of the Marcinkiewicz-Zygmund inequality called Khintchine's inequality which can be used in the proof of the former.
\begin{theorem}[Khintchine's inequality] Suppose that $p\in [2,\infty)$ and we are given independent random variables $X_1,\dots,X_n \in L^p(\P)$ with $\P(X_i=a_i)=\P(X_i=-a_i)=1/2$.  Then
\begin{equation*}
\|\sum_{i}{X_i}\|_{L^p(\P)} = O\left(\sqrt{p}\|\sum_{i}{|X_i|^2}\|_{L^{p/2}(\P)}^{1/2}\right)= O\left(\sqrt{p}\left(\sum_{i}{|a_i|^2}\right)^{1/2}\right).
\end{equation*}
\end{theorem}
Khintchine's inequality is proved by restricting to the case when $p$ is an even integer (the other cases follow by nesting of norms) and then raising the left hand side to the power $p$, multiplying it out and collecting together terms.  There are more elegant proofs but this gives the main idea.

Given this, to prove the Marcinkiewicz-Zygmund inequality one can proceed by a process of symmetrisation.  First, if the variables are complex then the result follows from taking real and imaginary parts and so one may as well assume they are real.  We then take copies $Y_1,\dots,Y_n$ of $X_1,\dots,X_n$ such that $X_i \sim Y_i$ and $X_1,\dots,X_n,Y_1,\dots,Y_n$ are mutually independent.  Following this we apply Khintchine's inequality to the variables $X_i-Y_i$ restricted to atoms of the sample space on which they are symmetric and only take two values.  Collecting all this together gives the result.

\begin{proof}[Proof of Lemma \ref{lem.cs}]
Let $z_1,\dots,z_k$ be independent uniformly distributed $S$-valued random variables, and for each $y \in G$ define $Z_i(y):=\rho_{-z_i}(f)(y) - f \ast \mu_S(y)$.  For fixed $y$, the variables $Z_i(y)$ are independent and have mean zero, so it follows by the Marcinkiewicz-Zygmund inequality and H{\"o}lder's inequality that
\begin{eqnarray*}
\| \sum_{i=1}^k{Z_i(y)}\|_{L^p(\mu_S^k)}^p &\leq &O(p)^{p/2}\int{\left(\sum_{i=1}^k{|Z_i(y)|^2}\right)^{p/2}d\mu_S^k}\\ & \leq & O(p)^{p/2}k^{p/2-1}\sum_{i=1}^k{\int{|Z_i(y)|^p}d\mu_S^k}.
\end{eqnarray*}
Summing over $y$ and interchanging the order of summation we get
\begin{equation}\label{eqn.khin}
\sum_{y \in G}{\| \sum_{i=1}^k{Z_i(y)}\|_{L^p(\mu_S^k)}^p} \leq O(p)^{p/2}k^{p/2-1}\int{\sum_{i=1}^k{\sum_{y \in G}{|Z_i(y)|^p}}d\mu_{S}^k}.
\end{equation}
On the other hand,
\begin{equation*}
\left(\sum_{y\in G}{|Z_i(y)|^p}\right)^{1/p} =\|Z_i\|_{\ell^p(G)} \leq \|\rho_{-z_i}(f)\|_{\ell^p(G)} + \|f \ast \mu_S\|_{\ell^p(G)} \leq 2\|f\|_{\ell^p(G)}
\end{equation*}
by the triangle inequality.  Dividing (\ref{eqn.khin}) by $k^p$ and inserting the above and the expression for the $Z_i$s we get that
\begin{equation*}
\int{\sum_{y \in G}{\left|\frac{1}{k}\sum_{i=1}^k{\rho_{-z_i}(f)(y)} - f \ast \mu_S(y)\right|^p}d\mu_S^k(z)}=O(pk^{-1}\|f\|_{\ell^p(G)}^2)^{p/2}.
\end{equation*}
Pick $k=O(\eta^{-2}p)$ such that the right hand side is at most $(\eta \|f\|_{\ell^p(G)}/4)^p$ and write $\mathcal{L}$ for the set of $x \in S^k$ for which the integrand above is at most $(\eta \|f\|_{\ell^p(G)}/2)^p$; by averaging $\mu_S^k(\mathcal{L}^c) \leq 2^{-p}$ and so $\mu_S^k(\mathcal{L}) \geq 1-2^{-p} \geq 1/2$.

Now, $\Delta:=\{(t,\dots,t): t \in T\}$ has $\mathcal{L}+\Delta  \subset (S+T)^k$, whence $|\mathcal{L} + \Delta | \leq 2L^k|\mathcal{L}|$ and so
\begin{equation*}
\langle 1_\Delta \ast 1_{-\Delta},1_{-\mathcal{L}} \ast 1_{\mathcal{L}}\rangle_{\ell^2(G^k)} = \|1_\mathcal{L} \ast 1_\Delta \|_{\ell^2(G^k)}^2\geq |\Delta|^2|\mathcal{L}|/2L^k,
\end{equation*}
by the Cauchy-Schwarz inequality.

By averaging it follows that at least $|\Delta|^2/2L^k$ pairs $(z,y) \in \Delta\times\Delta$ have $1_{-\mathcal{L}} \ast 1_{\mathcal{L}}(z-y)>0$, and hence there are at least $|\Delta|/2L^k=|T|/2L^k$ distinct elements $x \in T-T \subset G$ with $1_{-\mathcal{L}} \ast 1_{\mathcal{L}}(x,\dots,x)>0$; write $X$ for this set.

By design for each $x \in X$ there is some $z(x) \in \mathcal{L}$ and $y(x) \in \mathcal{L}$ such that $y(x)_i=z(x)_i+x$.  But then by the triangle inequality we get that
\begin{eqnarray*}
\|\rho_{-x}(f \ast \mu_S) - f \ast \mu_S\|_{\ell^p(G)}& \leq &\|\frac{1}{k}\sum_{i=1}^k{\rho_{-y(x)_i}(f)} -  f \ast \mu_S\|_{\ell^p(G)}\\&&+\|\rho_{-x}\left(\frac{1}{k}\sum_{i=1}^k{-\rho_{z(x)_i}(f)} - f \ast \mu_S \right)\|_{\ell^p(G)}.
\end{eqnarray*}
However, since $\rho_x$ is isometric on $\ell^p(G)$ we see that
\begin{eqnarray*}
\|\rho_{-x}(f \ast \mu_S) - f \ast \mu_S\|_{\ell^p(G)} &\leq & \|\frac{1}{k}\sum_{i=1}^k{\rho_{-y(x)_i}(f)}-  f \ast \mu_S\|_{\ell^p(G)}\\&&+\|\frac{1}{k}\sum_{i=1}^k{\rho_{-z(x)_i}(f)} - f \ast \mu_S\|_{\ell^p(G)} \leq 2(\eta\|f\|_{\ell^p(G)}/2),
\end{eqnarray*}
since $z(x),y(x) \in \mathcal{L}$.
\end{proof}
The real strength here is the quality of the bounds for large $p$.  For $p=2$ a stronger result follows from Chang's theorem (at least in the case of good modelling in the sense of Green and Ruzsa \cite{greruz::0}) which can actually be used to show that the set on which $f \ast \mu_S$ is approximately invariant is not just large, but it actually contains a large Bohr set.  The techniques for proving this are Fourier analytic in nature and yield doubly exponential dependence on $p$ if they are used to prove a version of the above result.

In the next section we shall make more careful use of the above result for large $p$, but here we just use the $p=2$ case to give a set of polynomial growth following the outline in the previous section.
\begin{proposition*}[Proposition \ref{prop.oldstateofart}]
Suppose that $|A+A| \leq K|A|$.  Then $A$ is $\exp(O(K\log K))$-covered by a symmetric neighbourhood of the identity of size at most $\exp(O(\log K))|A|$ and relative polynomial growth of order $O(K\log^3K)$.
\end{proposition*}
\begin{proof}
We put $f=1_{A}$ and apply the Croot-Sisask lemma with $p=2$, $S=T=A$, and a parameter $\eta/m$ (where $\eta$ and $m$ are to be optimised later) to get a symmetric neighbourhood of the identity, $X$, with $|X| \geq (2K)^{-O(m^2\eta^{-2})}|A|$ such that
\begin{equation*}
\|\rho_x(1_A \ast \mu_A) - 1_A \ast \mu_A\|_{\ell^2(G)}^2 \leq \eta^2m^{-2} |A| \text{ for all } x \in X.
\end{equation*}
It follows by the triangle inequality that
\begin{equation*}
\|\rho_x(1_A \ast \mu_A) - 1_A \ast \mu_A\|_{\ell^2(G)}^2 \leq \eta^2 |A| \text{ for all } x \in mX,
\end{equation*}
and then multiplying out the $\ell^2$-norm we see that
\begin{equation*}
2\|1_A \ast \mu_A\|_{\ell^2(G)}^2 - 2\langle \rho_x(1_A \ast \mu_A), 1_A \ast \mu_A \rangle_{\ell^2(G)} \leq \eta^2|A|.
\end{equation*}
Of course by the Cauchy-Schwarz inequality we have that
\begin{equation*}
\|1_A \ast \mu_A\|_{\ell^2(G)}^2\geq  \frac{1}{|A+A|}\|1_A \ast \mu_A\|_{\ell^1(G)}^2 \geq |A|/K,
\end{equation*}
thus if we set $\eta^2=1/K$ we get
\begin{equation*}
\langle \rho_x(1_A \ast \mu_A), 1_A \ast \mu_A \rangle_{\ell^2(G)} \geq |A|/2K.
\end{equation*}
It follows that $x \in 2A-2A$, and so $mX \subset 2A-2A$.  Now by Pl{\"u}nnecke's inequality we have that $|(3l+1)(2A-2A)| \leq K^{4(3l+1)}|A|$ and so
\begin{equation*}
|(3ml+1)X|\leq |(3l+1)mX| \leq (2K)^{O(l+m^2K)}|X|.
\end{equation*}
We put $l=m^2K+O(1)$ and write $k:=ml=m^3K+O(m)$ so that
\begin{equation*}
|(3k+1)X| \leq (2K)^{O(m^2K)}|X|=\exp(O(km^{-1} \log K))|X|.
\end{equation*}
We can then pick $m=O(\log K)$ such that the right hand side is strictly less than $2^k|X|$ and hence $|(3k+1)X| < 2^k|X|$.  Thus by Corollary \ref{cor.useful} we have that $X$ has relative polynomial growth of order $O(k)=O(K\log^3K)$.

On the other hand, since $X$ is symmetric we have $X-X \subset 2A-2A$ and so $|X-X| \leq K^4|A|$ by the Pl{\"u}nnecke-Ruzsa inequalities, but also $X+A \subset 3A-2A$.  Of course with these choices $|X| \geq \exp(-O(K\log K))|A|$ and hence $|X+A| \leq \exp(O(K\log K))|X|$ by the Pl{\"u}nnecke-Ruzsa inequalities.  With this information Lemma \ref{lem.convert} completes the proof.
\end{proof}
This  result gives bounds of roughly the same order as those of Green and Ruzsa \cite{greruz::0}, and more or less represents the state of the art prior to Schoen's work \cite{sch::1}.

\section{A weak Bogolyubov-Ruzsa-type lemma with strong bounds}\label{sec.kony}

This section contains most of the newest material and we shall start with a proof of an asymmetric weak Bogolyubov-Ruzsa-type lemma with good bounds in line with the aims of \S\ref{sec.wb}.

Before diving in it is worth making a few motivating remarks.  Our starting point is the argument at the end of the last section (the proof of Proposition \ref{prop.oldstateofart}).  The weakness there was that we had to apply the Croot-Sisask lemma with a very small choice of $\eta$.  This was because we have the lower bound
\begin{equation*}
\|1_A \ast \mu_A\|_{\ell^2(G)}^2 \geq |A|/K
\end{equation*}
which is small when compared with the trivial upper bound of $|A|$.  We should like something somewhat larger, but as it is the lower bound may well be nearly this small.  In \cite{sch::1} Schoen addressed this problem by proving the following important combinatorial lemma.
\begin{lemma}[{\cite[Lemma 3]{sch::1}}]  Suppose that $|A+A| \leq K|A|$ and $\epsilon\in (0,1]$.  Then there are sets $X \subset A-A$ and $Y \subset A+A$ such that $|X| \geq \exp(-O(2^{\epsilon^{-1}}\log K))|A|$ and $|Y| \geq |A|$ such that
\begin{equation*}
\|1_Y \ast \mu_X\|_{\ell^2(G)}^2 \geq K^{-2\epsilon}|Y|.
\end{equation*}
\end{lemma}
The proof of this is a beautiful induction using an observation of Katz and Koester \cite{katkoe::}, which we shall not, unfortunately, have time to pursue here.

Given this lemma we proceed along the lines of the proof of Proposition \ref{prop.oldstateofart} but using the sets $Y$ and $X$ given by the lemma instead of $A$ and this yields the following proposition.
\begin{proposition}
Suppose that $|A+A| \leq K|A|$.  Then $A$ is $\exp(\exp(O(\sqrt{\log K})))$-covered by a symmetric neighbourhood of the identity of size at most $\exp(O(\log K))|A|$ and relative polynomial growth of order $\exp(O(\sqrt{\log K}))$.
\end{proposition}

Our approach here is somewhat different and instead of taking the inner product of $1_A \ast 1_A$ (or $1_Y \ast 1_X$) with itself we take a different function following L{\'o}pez and Ross \cite{lopros::}:
\begin{equation*}
\langle 1_A \ast \mu_A , 1_{A+A}\rangle = |A|.
\end{equation*}
Given the above identity we should like to analyse $1_{A+A} \ast \mu_{A}$ using the Croot-Sisask lemma; we do this now in the more convenient case of symmetric sets although the argument is not essentially different.
\begin{proposition}\label{prop.key}
Suppose that $S \subset G$ is symmetric and $|S+S| \leq K|S|$, $T$ has $|S+T| \leq L|S|$, and $m \in \N$ is a parameter.  Then there is a symmetric neighbourhood of the identity, $X$, with
\begin{equation*}
|X| \geq \exp(-O(m^2\log K \log L))|T| \text{ and } mX \subset 4S.
\end{equation*}
\end{proposition}
\begin{proof}
We put $f=1_{S+S}$ and apply the Croot-Sisask lemma with a parameter $\eta$ (to be optimised later) to get a symmetric neighbourhood of the identity, $X$, with $|X| \geq (2L)^{-O(\eta^{-2}p)}|T|$ such that
\begin{equation*}
\|\rho_x(1_{S+S} \ast \mu_S) - 1_{S+S}\ast \mu_S\|_{\ell^p(G)} \leq \eta \|1_{S+S}\|_{\ell^p(G)} \text{ for all } x \in X.
\end{equation*}
It follows by the triangle inequality that
\begin{equation*}
\|\rho_x(1_{S+S} \ast \mu_S) - 1_{S+S}\ast \mu_S\|_{\ell^p(G)} \leq \eta m \|1_{S+S}\|_{\ell^p(G)} \text{ for all } x \in mX.
\end{equation*}
Taking an inner product with $\mu_S$ we see that
\begin{equation*}
|\langle \rho_x(1_{S+S} \ast \mu_S),\mu_S\rangle - \langle 1_{S+S} \ast \mu_S,\mu_S\rangle| \leq \eta m \|1_{S+S}\|_{\ell^p(G)}\|\mu_S\|_{\ell^{p'}(G)}
\end{equation*}
where $p'$ is the conjugate exponent to $p$.  Now
\begin{equation*}
 \langle 1_{S+S} \ast \mu_S,\mu_S\rangle = \langle 1_{S+S},\mu_S\ast \mu_{S}\rangle = 1,
\end{equation*}
since $S$ is symmetric and $\supp \mu_S \ast\mu_{S} \subset S+S$.  Thus
\begin{equation*}
|\mu_{S} \ast 1_{S+S} \ast \mu_S(x) - 1| \leq \eta m \|1_{S+S}\|_{\ell^p(G)}\|\mu_S\|_{\ell^{p'}(G)} \leq \eta m K^{1/p}.
\end{equation*}
We take $p = 2+\log K$, and then $\eta = \Omega(m^{-1})$ such that the term on the right is at most $1/2$ to get the desired conclusion.
\end{proof}
As a consequence of this we already get the following poly-logarithmic bounds.
\begin{proposition*}[Proposition \ref{prop.grow1}]
Suppose that $|A+A| \leq K|A|$.  Then $A$ is $\exp(O(\log^{4}K))$-covered by a symmetric neighbourhood of the identity of size at most $\exp(O(\log K))|A|$ and relative polynomial growth of order $O(\log^{4}K)$.
\end{proposition*}
\begin{proof}
We apply the previous result with $T=S=A-A$ and a parameter $m \in \N$ to be optimised later to get a symmetric neighbourhood of the identity, $X$, with
\begin{equation*}
|X| \geq \exp(-O(m^2\log^2K))|A-A| \text{ and } mX \subset 4(A-A).
\end{equation*}
Given $l \in \N$ also to be optimised later, by the Pl{\"u}nnecke's inequality we have that
\begin{equation*}
|(3ml+1)X| \leq |(3l+1)4(A-A)| \leq K^{O(l)}|A-A| \leq K^{O(l)}\exp(O(m^2\log^2K))|X|.
\end{equation*}
We now put $l=m^2\log K+O(1)$ and write $k:=ml=m^3\log K +O(m)$ so that we have
\begin{equation*}
|(3k+1)X| \leq \exp(O(m^2\log^{2}K))|X| = \exp(O(km^{-1}\log K))|X|.
\end{equation*}
We can then pick $m=O(\log K)$ such that the right hand side is strictly less than $2^k|X|$ and hence $|(3k+1)X| < 2^k|X|$.  Thus by Corollary \ref{cor.useful} we have that $X$ has relative polynomial growth of order $O(k)=O(\log^4K)$.

On the other hand we have $X-X \subset 4A-4A$ and so $|X-X| \leq K^8|A|$ by the Pl{\"u}nnecke-Ruzsa inequalities, but also $X+A \subset 5A-4A$.  Hence $|X+A| \leq \exp(O(\log^4K))|X|$ by the Pl{\"u}nnecke-Ruzsa inequalities.  With this information Lemma \ref{lem.convert} completes the proof.
\end{proof}
We saw Proposition \ref{prop.key} with $S=T$ is already rather powerful, but Konyagin introduced a rather nice bootstrapping technique whereby the result is first applied iteratively to reduce $L$ to $O(1)$.  To do this we first note the following corollary of Proposition \ref{prop.key}.
\begin{corollary}
Suppose that $S \subset G$ is a symmetric neighbourhood of the identity and $|S+S| \leq K|S|$, $T$ is a symmetric neighbourhood of the identity with $|S+T| \leq L|S|$, and $D\geq 1$ is a parameter.  Then there is some symmetric neighbourhood of the identity, $T'$, such that
\begin{equation*}
|T'| \geq \exp(-O(D^2\log L\log K))|T|,
\end{equation*}
and a symmetric neighbourhood of the identity $S'$ with $S \subset S' \subset 5S$ and $|S'+T'| \leq K^{1/D}|S'|$.
\end{corollary}
\begin{proof}
Let $k$ be a natural number to be optimised later and apply Proposition \ref{prop.key} to get a symmetric neighbourhood of the identity, $T'$, such that
\begin{equation*}
|T'| \geq \exp(-O(k^2\log L \log K))|T| \text{ and } 4S \supset kT'.
\end{equation*}
It follows that $|S+kT'| \leq |5S| \leq K^5|S|$ by Pl{\"u}nnecke's inequality.  Thus by the pigeon-hole principle there is some $l \in \{0,\dots,k-1\}$ such that
\begin{equation*}
|(S+lT') + T'| \leq K^{5/k}|S+lT'|.
\end{equation*}
Of course we can pick $k=O(D)$ such that $K^{5/k} \leq K^{1/D}$ and so putting $S':=S+lT'$ the corollary is proved.
\end{proof}
The pigeon-holing trick was developed by Tao in \cite{tao::9} to establish a Fre{\u\i}man-type result in the non-Abelian setting but has since found use in the Abelian setting.

We are now in a position to apply the above corollary iteratively.
\begin{proposition}\label{prop.kony}
Suppose that $|A+A| \leq K|A|$.  Then there is some symmetric neighbourhood of the identity, $T$, a natural number $r=O(\log \log K)^{O(1)}$, and a symmetric neighbourhood of the identity $S$ with $A-A \subset S \subset r(A-A)$ and
\begin{equation*}
|S + T| =O(|S|) \text{ and } |T| \geq \exp(-O(\log \log K)^{O(1)}\log^3 K)|S|.
\end{equation*}
\end{proposition}
\begin{proof}
We define two sequences of sets $(S_i)_i$ and $(T_i)_i$, and a sequence of reals $(L_i)_i$ such that $S_i$ and $T_i$ are symmetric neighbourhoods of the identity, and
\begin{equation*}
A-A \subset S_i \subset 5^i(A-A) \text{ and } |S_i+T_i| \leq L_i|S_i|,
\end{equation*}
where $L_i=\exp(4(\log 2K)^{2^{-i}})$.  To start with we put $S_0:=A-A$ and $T_0:=A-A$ which satisfies the requirements by the Pl{\"u}nnecke-Ruzsa inequalities.  At stage $i$ we note that
\begin{equation*}
|S_i+S_i| \leq |2 \cdot 5^i(A-A)| \leq K^{4\cdot 5^i}|A-A| \leq K^{4\cdot 5^i}|S_i|
\end{equation*}
by the Pl{\"u}nnecke-Ruzsa inequalities.  We apply the previous corollary to the sets $S_i$ and $T_i$ with parameter $D_i:=1+(\log(2K^{4 \cdot 5^i}))^{1-2^{-(i+1)}}$ to get symmetric neighbourhoods of the identity $S_{i+1}$ and $T_{i+1}$, with
\begin{eqnarray*}
|T_{i+1}| &\geq &\exp(-O(D_i^2(\log L_i) (\log  K^{4\cdot 5^i})))|T_i|\geq \exp(-O(\exp(O(i))\log^3 K))|T_i|,
\end{eqnarray*}
\begin{equation*}
A-A \subset S_i \subset S_{i+1} \subset 5S_i \subset 5^{i+1}(A-A) 
\end{equation*}
and
\begin{equation*}
|S_{i+1} + T_{i+1}| \leq \exp(4(\log2K)^{2^{-(i+1)}})|S_{i+1}|.
\end{equation*}
We terminate the iteration when $2^{i+O(1)} = \log 2\log 2K$ and find that the result is proved with $S=S_i$ and $T=T_i$.
\end{proof}
Finally we have the strongest result of the section and the driving ingredient in this survey.
\begin{proposition*}[Proposition \ref{prop.grow2}]
Suppose that $|A+A| \leq K|A|$.  Then $A$ is $\exp(O(\log^{3+o(1)}K))$-covered by a symmetric neighbourhood of the identity  of size at most $\exp(O(\log^{1+o(1)}K))|A|$ and relative polynomial growth of order $O(\log^{3+o(1)}K)$.
\end{proposition*}
\begin{proof}
We apply Proposition \ref{prop.kony} to the set $A$ to get symmetric neighbourhoods of the identity $S$ and $T$, and a natural number $r=O(\log^{o(1)}K)$ such that
\begin{equation*}
A-A \subset S \subset r(A-A), |S+T| =O(|S|) \text{ and } |T| \geq \exp(-O(\log^{3+o(1)}K))|S|.
\end{equation*}
Now, by Proposition \ref{prop.key} applied to the sets $S$ and $T$ with a parameter $m$ to be optimised later we get a symmetric neighbourhood of the identity $X$ with
\begin{equation*}
|X| \geq \exp(-O(m^2\log^{1+o(1)} K))|T| \text{ and } mX \subset 4S.
\end{equation*}
Given $l \in \N$ also to be optimised later, by Pl{\"u}nnecke's inequality we have that
\begin{eqnarray*}
|(3ml+1)X| &\leq &|(3l+1)4S| \leq K^{O(l)}|S|\\ & \leq & K^{O(l)}\exp(O(m^2\log^{1+o(1)}K) + O(\log^{3+o(1)}K))|X|
\end{eqnarray*}
We now put $l=m^2\log^{o(1)} K$ and write $k:=ml=m^3\log^{o(1)}K$ so that we have 
\begin{equation*}
|(3k+1)X| \leq \exp(O(k(m^{-1}\log^{1+o(1)}K+m^{-3}\log^{3+o(1)}K))|X|.
\end{equation*}
We can then pick $m=\log^{1+o(1)} K$ such that the right hand side is strictly less than $2^k|X|$ and hence $|(3k+1)X| < 2^k|X|$.  Thus by Corollary \ref{cor.useful} we have that $X$ has relative polynomial growth of order $O(k)=O(\log^{3+o(1)}K)$.

On the other hand we have $X-X \subset 4S \subset 4r(A-A)$ and so, by the Pl{\"u}nnecke-Ruzsa inequalities, we have
\begin{equation*}
|X-X| \leq |4rA - 4rA| \leq K^{8r}|A| \leq \exp(O(\log^{1+o(1)}K))|A|.
\end{equation*}
This set inclusion (and the fact that $0_G \in X$) also tells us that $X+A \subset 4r(A-A)+A$.  Hence, by the Pl{\"u}nnecke-Ruzsa inequalities again, and the fact that $|A| \leq |S| \leq \exp(O(\log^{3+o(1)}K))|X|$ we have
\begin{equation*}
|X+A| \leq K^{8r+1}|A| = \exp(O(\log^{1+o(1)}K))|A| \leq \exp(O(\log^{3+o(1)}K))|X|
\end{equation*}
With this information Lemma \ref{lem.convert} completes the proof.
\end{proof}
It may be worth saying that all the $\log^{o(1)}K$ terms in the above proposition can be replaced by $(\log \log K)^{O(1)}$ terms if desired.

\section{From relative polynomial growth to convex coset progressions}\label{sec.rpgccp}

Our aim in the next few sections it to prove Theorem \ref{thm.rtc} which we restate now for convenience.
\begin{theorem*}[Theorem \ref{thm.rtc}]
Suppose that $X$ has relative polynomial growth of order $d$.  Then there is a centred convex coset progression $M$ such that
\begin{equation*}
X-X \subset M, \dim M =O(d\log^2 d) \text{ and } |M| \leq \exp(O(d\log^2 d))|X|.
\end{equation*}
\end{theorem*}
We shall make considerable use of harmonic analysis on discrete groups to do this and so it will be useful to record some definitions.  The classic reference is Rudin \cite{rud::1} although the reader will be equally well served by Tao and Vu \cite{taovu::}.

We have already introduced convolution, and the Fourier transform is defined to diagonalise the operators induced by convolution, so we are already have quite a bit of what we need.

Given $G$ (discrete) we write $\wh{G}$ for the set of homomorphisms $\gamma:G \rightarrow S^1$ where $S^1:=\{z \in \C:|z|=1\}$.  These homomorphisms are called \textbf{characters} and the set $\wh{G}$ naturally supports the structure of a topological group, in particular a compact Abelian group under point-wise multiplication of characters, called the \textbf{dual group} of $G$.

The dual group is naturally endowed with a translation invariant probability measure called the Haar probability measure and we are now in a position to define the Fourier transform.  Given $f \in \ell^1(G)$ we define the \textbf{Fourier transform} of $f$ to be the function $\wh{f} \in L^\infty(\wh{G})$ determined by
\begin{equation*}
\wh{f}(\gamma):=\sum_{x \in G}{f(x)\overline{\gamma(x)}} \text{ for all } \gamma \in \wh{G}.
\end{equation*}
This has the property that $\wh{f \ast g} = \wh{f} \cdot \wh{g}$.  More than this we have Plancherel's formula which tells us that
\begin{equation*}
\langle f,g\rangle_{\ell^2(G)} = \langle \wh{f},\wh{g}\rangle_{L^2(\wh{G})} \text{ for all }f,g \in \ell^2(G).
\end{equation*}
We have already indicated that $\wh{G}$ has a natural topology, and in fact if $G$ is small enough this topology is induced by a metric.  There are then a range of metrics which define different topologies of $\wh{G}$ reflecting the subgroup structure of $\wh{G}$.  These can be defined by bases of what are called Bohr sets.

Given a neighbourhood $\Gamma$ of characters on $G$ and a parameter $\delta \in (0,2]$ we define the \textbf{Bohr set} with \textbf{frequency set } $\Gamma$ and \textbf{width} $\delta$ to be the set
\begin{equation*}
\Bohr(\Gamma,\delta):=\{x \in G: |\gamma(x) - 1| \leq \delta \text{ for all } \gamma \in \Gamma\}.
\end{equation*}
One rather useful property of Bohr sets which we use repeatedly is the fact that they are balls in a pseudo-metric.  What we mean by this is that for a character $\gamma \in \wh{G}$ we have the very useful triangle inequality
\begin{equation*}
|1-\gamma(x+y)| = |1-\gamma(x) + (1-\gamma(y))\gamma(x)| \leq |1-\gamma(x)| + |1-\gamma(y)|
\end{equation*}
for all $x,y \in G$.

The first ingredient in proving Theorem \ref{thm.rtc} is to show that in some sense the topology determined by a set $X$ is roughly the same as that  determined by certain Bohr sets.
\begin{proposition}\label{prop.polytobohr}
Suppose that $X$ has relative polynomial growth of order $d$.  Then there is a neighbourhood of characters $\Gamma$ and a natural number $k=O(d\log^2 d)$ such that
\begin{equation*}
X-X \subset \Bohr(\Gamma,1/(4(3k+1))) \text{ and }|\Bohr(\Gamma,1/2)| < 2^k|X|.
\end{equation*}
\end{proposition}
Now we shall see later that Bohr sets are already convex progressions, and if they satisfy a certain growth condition of the form used in Chang's covering lemma then they turn out to be low-dimensional.  In particular we have shall show the following which combines with the previous result to yield Theorem \ref{thm.rtc}.
\begin{proposition}\label{prop.coset}
Suppose that $\Bohr(\Gamma,\delta)$ is a finite Bohr set and $k \in \N$ is such that
\begin{equation*}
|\Bohr(\Gamma,(3k+1)\delta)|<2^k|\Bohr(\Gamma,\delta)| \text{ for some } \delta < 1/(4(3k+1)).
\end{equation*}
Then $\Bohr(\Gamma,\delta)$ is an (at most) $k$-dimensional centred convex coset progression.
\end{proposition}

\section{Relative polynomial growth and Bohr sets}\label{sec.ccprpg}

In this section we show how to pass from sets with relative polynomial growth to a Bohr set which (effectively) has polynomial growth of relatively low order.  Shortly we shall see that Bohr sets are convex coset progressions (provided the width parameter is sufficiently small), but for now we think of them as a sort of `approximate annihilator'.

To find an appropriate Bohr set we shall need to examine the (very) large spectrum of a finite set $A$, which is defined to be the set
\begin{equation*}
\LSpec(A,\epsilon):=\{\gamma \in \wh{G}: \|1-\gamma\|_{L^2(\mu_A \ast \mu_{-A})} \leq \epsilon\}.
\end{equation*}
(Note immediately that $\LSpec(A,\epsilon)$ is a neighbourhood since $A$ is finite.)  The definition of $\LSpec$ we have given takes the form it does for ease of use of the triangle inequality: if $\gamma \in \LSpec(A,\epsilon)$ and $\gamma' \in \LSpec(A,\epsilon')$ then $\gamma+\gamma' \in \LSpec(A,\epsilon+\epsilon')$ by the triangle inequality:
\begin{eqnarray*}
\|1-\gamma\gamma'\|_{L^2(\mu_A \ast \mu_{-A})} &= &\|(1-\gamma') + (1-\gamma)\gamma'\|_{L^2(\mu_A \ast \mu_{-A})}\\& \leq & \|1-\gamma'\|_{L^2(\mu_A \ast \mu_{-A})} + \|(1-\gamma)\gamma'\|_{L^2(\mu_A \ast \mu_{-A})}\\ &= & \|1-\gamma\|_{L^2(\mu_A \ast \mu_{-A})}+\|1-\gamma'\|_{L^2(\mu_A \ast \mu_{-A})}.
\end{eqnarray*}
On the other hand to connect the definition to the idea that $\LSpec$ should represent the large spectrum we have the following useful identity:
\begin{equation*}
\|1-\gamma\|_{L^2(\mu_A \ast \mu_{-A})}^2=2(1-|\wh{\mu_A}(\gamma)|^2),
\end{equation*}
so that
\begin{equation*}
\|1-\gamma\|_{L^2(\mu_A \ast \mu_{-A})}\leq \epsilon \text{ if and only if } |\wh{\mu_A}(\gamma)| \geq \sqrt{1-\epsilon^2/2}.
\end{equation*}
This fact will be used extensively in the remainder of the section.

We have two key tools for establishing our main proposition (Proposition \ref{prop.polytobohr}).  The first of these uses an approximation developed by Schoen in \cite{sch::0} and imported into this context by Green and Ruzsa in \cite{greruz::0}.
\begin{proposition}\label{prop.upper}
Suppose that $X$ has relative polynomial growth of order $d$.  Then 
\begin{equation*}
|\Bohr(\LSpec(X,\epsilon),1/2)| \leq \exp(O(d\log \epsilon^{-1}d))|X|.
\end{equation*}
\end{proposition}
\begin{proof}
By Plancherel's theorem and the Cauchy-Schwarz inequality we have
\begin{equation}\label{eqn.polynomialgrowthCS}
\int{|\wh{1_X}(\gamma)|^{2k}d\gamma} = \|1_X^{(k)}\|_{\ell^2(G)}^2 \geq\frac{\|1_X^{(k)}\|_{\ell^1(G)}^2}{|\supp 1_X^{(k)}|} = \frac{|X|^{2k}}{|kX|}.
\end{equation}
We shall show that most of this mass is supported on the set of characters where the Fourier transform of $1_X$ is very large.  In particular note that
\begin{eqnarray*}
\int_{ \LSpec(X,\epsilon)^c}{|\wh{1_X}(\gamma)|^{2k}d\gamma} & \leq & (\sqrt{1-\epsilon^2/2}|X|)^{2k-2}\int{|\wh{1_X}|^2d\gamma}\\
& = & (1-\epsilon^2/2)^{k-1}|X|^{2k-1},
\end{eqnarray*}
by Parseval's theorem.

Since $X$ has polynomial growth of order $d$ we have that $|kX| \leq k^d|X|$ for $k \geq 1$, so there is a positive
integer $k$ with $ k=O(\epsilon^{-2}d\log \epsilon^{-1}d)$ and
\begin{equation*}
(1-\epsilon^2/2)^{k-1} \leq 1/2k^d \leq |X|/2|kX|,
\end{equation*}
whence 
\begin{equation*}
\int_{\LSpec(X,\epsilon)^c}{|\wh{1_{X}}(\gamma)|^{2k}d\gamma} \leq \frac{|X|^{2k}}{2|kX|}.
\end{equation*}
Thus, by (\ref{eqn.polynomialgrowthCS}) we have
\begin{equation*}
\int_{\LSpec(X,\epsilon)}{|\wh{1_X}(\gamma)|^{2k}d\gamma} \geq \frac{|X|^{2k}}{2|kX|}.
\end{equation*}

Now, let $B$ be a finite subset of $\Bohr(\LSpec(X,\epsilon),1/2)$.  Integrating we get that $|1-\wh{\mu_B}(\gamma)| \leq 1/2$ for any $\gamma \in \LSpec(X,\epsilon)$ and it follows by the triangle inequality that $|\wh{\mu_B}(\gamma)| \geq 1/2$.  Consequently
\begin{equation*}
\int{|\wh{1_{X}}(\gamma)|^{2k}|\wh{\mu_B}(\gamma)|^2d\gamma} \geq 2^{-2}\int_{\LSpec(X,\epsilon)}{|\wh{1_{X}}(\gamma)|^{2k}d\gamma}\geq \frac{|X|^{2k}}{2^3|kX|}.
\end{equation*}
On the other hand
\begin{eqnarray*}
\int{|\wh{1_{X}}(\gamma)|^{2k}|\wh{\mu_B}(\gamma)|^2d\gamma} & \leq & |X|^{2k-2}\|1_X \ast \mu_B\|_{\ell^2(G)}^2\\
& \leq & |X|^{2k-2}\|1_X \ast \mu_B\|_{\ell^1(G)} \|1_X \ast \mu_B\|_{\ell^\infty(G)}
\end{eqnarray*}
by the Hausdorff-Young inequality, Parseval's theorem and then H\"{o}lder's inequality. Since $\|1_{X}\ast \mu_B\|_{\ell^1(G)}=|X|$ we conclude that
\begin{equation*}
\frac{|X|}{2^3|kX|} \leq \|1_X \ast \mu_B\|_{\ell^\infty(G)} \leq \frac{|X|}{|B|}.
\end{equation*}
This gives the desired upper bound, but on $B$ rather than $\Bohr(\LSpec(X,\epsilon),1/2)$.  The result follows since $B$ was an arbitrary finite subset of $\Bohr(\LSpec(X,\epsilon),1/2)$.
\end{proof}
Our second key tool is yet another of the developments of Green and Ruzsa from \cite{greruz::0}.  It is only slightly more general than
\cite[Proposition 4.39]{taovu::}.
\begin{proposition}\label{prop.lowerbound}
Suppose that $|X+S| \leq K|S|$ and $\epsilon \in (0,1]$ is a parameter. Then
\begin{equation*}
X-X \subset \Bohr(\LSpec(X+S,\epsilon),O(\epsilon\sqrt{K})).
\end{equation*}
\end{proposition}
\begin{proof}
Write $\delta = 1-\sqrt{1-\epsilon^2/2}$ and suppose that $\gamma \in \LSpec(X+S,\epsilon)$. Then there is a phase $\omega \in
S^1$ such that
\begin{equation*}
\sum_{x \in G}{1_{X+S}(x)\omega\gamma(x)} = \omega\wh{1_{X+S}}(\gamma) = |\wh{1_{X+S}}(\gamma)|.
\end{equation*}
Since the right hand side is real we conclude that
\begin{equation*}
\sum_{x \in G}{1_{X+S}(x)\Re \omega\gamma(x)} = \Re \sum_{x \in G}{1_{X+S}(x)\omega\gamma(x)} = |\wh{1_{X+S}}(\gamma)| \geq (1-\delta)|X+S|.
\end{equation*}
It follows that
\begin{equation*}
\sum_{x \in G}{1_{X+S}(x)|1-\omega\gamma(x)|^2} =2\sum_{x \in G}{1_{X+S}(x)(1-\Re \omega\gamma(x))} \leq 2\delta |X+S|.
\end{equation*}
If $y_0,y_1 \in X$ then
\begin{equation*}
\sum_{x \in G}{1_{S}(x)|1-\omega\gamma(y_i)\gamma(x)|^2} \leq \sum_{x \in G}{1_{X+S}(x)|1-\omega\gamma(x)|^2} \leq 2\delta|X+S|.
\end{equation*}
The $2$-variable Cauchy-Schwarz inequality applied to $1-\omega\gamma(y_0)\gamma(x)$ and $1-\omega\gamma(y_1)\gamma(x)$ tells us that
\begin{eqnarray*}
|1-\gamma(y_0-y_1)|^2&=&|(1-\omega\gamma(y_0)\gamma(x)) - (1-\omega\gamma(y_1)\gamma(x))|^2\\ & \leq & 2(|1-\omega\gamma(y_0)\gamma(x)|^2 +|1-\omega\gamma(y_1)\gamma(x)|^2)
\end{eqnarray*}
for all $x \in G$ since $|\omega|=1$ and $|\gamma(x)|=1$, whence
\begin{equation*}
|S||1-\gamma(y_0-y_1)|^2 = \sum_{x \in G}{1_{S}(x)|1-\gamma(y_0-y_1)|^2} \leq 2^3\delta|X+S|.
\end{equation*}
The result follows since $\delta = O(\epsilon^2)$.
\end{proof}
With these two results we are in a position to prove the main result of this section.
\begin{proposition*}[Proposition \ref{prop.polytobohr}]
Suppose that $X$ has relative polynomial growth of order $d$.  Then there is a neighbourhood of characters $\Gamma$ and a natural number $k=O(d\log^2 d)$ such that
\begin{equation*}
X-X \subset \Bohr(\Gamma,1/(4(3k+1))) \text{ and }|\Bohr(\Gamma,1/2)| < 2^k|X|.
\end{equation*}
\end{proposition*}
\begin{proof}
Since $X$ has relative polynomial growth of order $d$ we may apply the pigeon-hole principle to pick $l = O(d \log d)$ such that $|X+ lX| = O(|lX|)$.  Let $\epsilon$ be a parameter to be optimised later.  By Proposition \ref{prop.upper} applied to the set $(l+1)X$ which has relative polynomial growth of order $O(d\log d)$ we see that for $\Gamma:=\LSpec(X+lX,\epsilon)$ (which is closed) we have
\begin{equation*}
|\Bohr(\Gamma,1/2)| \leq \exp(O(d\log^2 \epsilon^{-1}d))|X|.
\end{equation*}
On the other hand, by Proposition \ref{prop.lowerbound} applied to the sets $X$ and $lX$ we see that
\begin{equation*}
\Bohr(\Gamma,O(\epsilon)) \supset X-X.
\end{equation*}
We now pick $k=\Omega(\epsilon^{-1})$ such that the width parameter above is at most $1/(4(3k+1))$ and the size bound is less than $2^k$.  This is possible with $\epsilon = \Omega(1/(d\log^2 d))$.  The result is proved.
\end{proof}

\section{Ruzsa's embedding and convex coset progressions}\label{sec.lcp}

In the paper \cite{ruz::9} Ruzsa developed an important embedding for relating Bohr sets and convex coset progressions.  Given a set $\Gamma$ of characters on $G$, write $B(\Gamma,\R)$ for the vector space of bounded real-valued functions on $\Gamma$.  Now, we define the map
\begin{eqnarray*}
R_\Gamma:G& \rightarrow & B(\Gamma,\R)\\ x & \mapsto  & R_\Gamma(x):\Gamma \rightarrow \R; \gamma \mapsto \frac{1}{2\pi i}\log\gamma(x),
\end{eqnarray*}
where the logarithm takes its principal value.  (Since $|\gamma(x)| =1$ this means that the logarithm lies in $(-\pi i,\pi i]$ and so the functions are bounded.)

The map $R_\Gamma$ preserves inverses provided $\|R_\Gamma(x)\|_\infty < 1/2$, meaning that $R_\Gamma(-x)=-R_\Gamma(x)$; and furthermore we see that if
\begin{equation*}
\|R_\Gamma(x_1)\|_{\infty} + \dots + \|R_\Gamma(x_d)\|_{\infty} <1/2
\end{equation*}
then
\begin{equation*}
R_\Gamma(x_1+\dots + x_d)=R_\Gamma(x_1)+\dots+R_\Gamma(x_d).
\end{equation*}
This essentially encodes the idea that $R_\Gamma$ behaves like a Fre{\u\i}man morphism\footnote{We direct the unfamiliar reader to 
\cite[Chapter 5.3]{taovu::}.}, although we shall not formalise this notion here. We use this embedding to establish the following proposition.
\begin{proposition*}[Proposition \ref{prop.coset}]
Suppose that $\Bohr(\Gamma,\delta)$ is a finite Bohr set and $d \in \N$ is such that
\begin{equation*}
|\Bohr(\Gamma,(3d+1)\delta)|<2^d|\Bohr(\Gamma,\delta)| \text{ for some } \delta < 1/(4(3d+1)).
\end{equation*}
Then $\Bohr(\Gamma,\delta)$ is an (at most) $d$-dimensional centred convex coset progression.
\end{proposition*}
\begin{proof}
We shall prove that if $L:=\bigcap{\{\ker \gamma: \gamma \in \Gamma\}}$ is trivial then $\Bohr(\Gamma,\delta)$ is a $d$-dimensional centred convex progression.  The result then follows from this by quotienting out by $L$ (which does not impact the hypotheses of the proposition) to get a homomorphism $\phi:\Z^d \rightarrow G/L$ and a symmetric convex body $Q\subset \R^d$ such that $\Bohr(\Gamma,\delta)/L=\phi(Q\cap \Z^d)$.

Let $e_1,\dots,e_d$ be the standard set of generators for $\Z^d$ and for each $i \in \{1,\dots, d\}$ let $h_i \in G$ be a representative of $\phi(e_i)$.  Since $\Z^d$ is free define $\tilde{\phi}:\Z^d \rightarrow G$ by extension from its value at the generators $\tilde{\phi}(e_i):=h_i$ and note that
\begin{equation*}
\Bohr(\Gamma,\delta) =\bigcup{\Bohr(\Gamma,\delta)/L}=\bigcup{\phi(Q \cap \Z^d)} = \tilde{\phi}(Q \cap \Z^d)+L;
\end{equation*}
The result follows.

For notational convenience we write $B_\eta:=\Bohr(\Gamma,\eta)$ for any $\eta\in (0,2]$.  To start with note that if $x \in B_\eta$ then
\begin{equation*}
\|R_\Gamma(x)\|_{\infty} \leq \frac{1}{2\pi} \arccos (1-\eta^2/2) \leq 2\eta.
\end{equation*}
Since $2(3d+1)\delta < 1/2$ we have that if $x_1,\dots,x_{3d+1} \in B_\delta$ then
\begin{equation}\label{eqn.vfr}
R_\Gamma(x_1+\dots+x_{3d+1}) = R_\Gamma(x_1)+\dots + R_\Gamma(x_{3d+1}).
\end{equation}
By hypothesis we then have that
\begin{equation*}
|(3d+1)R_\Gamma(B_\delta)|  = |R_\Gamma((3d+1)B_\delta)| \leq |(3d+1)B_\delta|\leq |B_{(3d+1)\delta}| < 2^d|B_\delta|.
\end{equation*}
Now $|B_\delta| = |R_\Gamma(B_\delta)|$ since $R_\Gamma$ is injective on $B_\delta$.  To see this note that if $x,y \in B_\delta$ have $R_\Gamma(x)=R_\Gamma(y)$ then $R_\Gamma(x-y)=0$ by (\ref{eqn.vfr}) and the fact that $R_\Gamma$ preserves inverses on $B_\delta$.  It then follows that $\gamma(x-y)=1$ for all $\gamma \in \Gamma$, and since $L$ is trivial we conclude that $x=y$.

In light of all this we have that $|3dR_\Gamma(B_\delta) + R_\Gamma(B_\delta)| <2^d|R_\Gamma(B_\delta)|$, and so by the variant of Chang's covering lemma in Lemma \ref{lem.ccl} applied to the sets $3R_\Gamma(B_\delta)$ and $R_\Gamma(B_\delta)$ (both of which are symmetric neighbourhoods since $R_\Gamma$ preserves inverses and the identity, and $B_\delta$ is symmetric) we get a set $X\subset 3R_\Gamma(B_\delta)$ with $|X| < d$ such that
\begin{equation*}
3R_\Gamma(B_\delta) \subset \Span(X) + 2R_\Gamma(B_\delta)\subset \langle X\rangle + 2R_\Gamma(B_\delta).
\end{equation*}
Here, of course, $\langle X \rangle$ denotes the group generated by $X$.  It follows that for all $n \in \N$ we have
\begin{equation*}
(n+2)R_\Gamma(B_\delta)\subset \langle X\rangle + 2R_\Gamma(B_\delta).
\end{equation*}
Now, for each $v \in R_\Gamma(B_\delta)$ and $n \in \N$ there is some $v_n \in 2R_\Gamma(B_\delta)$ such that $nv \in \langle X \rangle + v_n$.  However, since $2R_\Gamma(B_\delta)$ is finite it follows that there are distinct natural numbers $n \neq m$ such that $v_n = v_m$ whence
\begin{equation*}
(n-m)v = nv-mv \in (\langle X \rangle + v_n) - (\langle X \rangle + v_m) = \langle X \rangle.
\end{equation*}
Thus to every $v \in R_\Gamma(B_\delta)$ there is some natural number $l_v$ such that $l_vv \in \langle X \rangle$.  Let $L$ be the lowest common multiple of all the natural numbers $(l_v)_{ v \in R_\Gamma(B_\delta)}$ so that $Lv \in \langle X \rangle$ for all $v \in R_\Gamma(B_\delta)$.  It follows that $v \in \langle x/L:x \in X\rangle$ and so $R_\Gamma(B_\delta)$ generates a lattice $\Lambda$ in $B(\Gamma,\R)$ of dimension $k \leq |X| < d$.

Let $v_1,\dots,v_k$ be a basis for $\Lambda$ and for each $j \in \{1,\dots,k\}$ write $v_j = \sum_{x \in B_\delta}{z_{j,x}R_\Gamma(x)}$ for some integers $(z_{j,x})_{x \in B_\delta}$.  We now put $h_j:= \sum_{x \in B_\delta}{z_{j,x}x}$ and define a homomorphism
\begin{equation*}
\phi:\Z^k \rightarrow G; (n_1,\dots,n_k) \mapsto n_1h_1+\dots+n_kh_k.
\end{equation*}
Finally write $V$ for the subspace of $B(\Gamma,\R)$ generated by $X$ and $\psi:V\rightarrow \R^k$ for the change of basis taking $v_i$ to the canonical basis vector $e_i$ of $\R^k$, and let $Q$ be the cube in $B(\Gamma,\R)$ centred at the origin and with side length $2\delta$.  The set $\psi(Q\cap V)$ is a symmetric convex body in $\R^k$ and it remains to check that $\phi(\psi(Q\cap V) \cap \Z^k)=B_\delta$.

If $x_0 \in B_\delta$ then $R_\Gamma(x_0) \in \Lambda$ and $R_\Gamma(x_0) \in Q$ and so
\begin{equation*}
R_\Gamma(x_0)=n_1v_1+\dots+n_kv_k\text{ for some } n \in \psi(Q \cap V) \cap \Z^k.
\end{equation*}
Given the definition of the $v_i$s we have that
\begin{equation*}
R_\Gamma(x_0) = \sum_{j=1}^k{n_j\sum_{x \in B_\delta}{z_{j,x}R_\Gamma(x)}}.
\end{equation*}
Exponentiating this point-wise (via $x \mapsto \exp(2\pi i x)$ which is a homomorphism from $B(\Gamma,\R) \rightarrow B(\Gamma,S^1)$) tells us that
\begin{equation*}
\gamma(x_0) = \prod_{j=1}^k{\left(\prod_{x \in B_\delta}{\gamma(x)^{z_{j,x}}}\right)^{n_j}} = \gamma(\sum_{j=1}^k{n_j\sum_{x \in B_\delta}{z_{j,x}x}}) \text{ for all } \gamma \in \Gamma.
\end{equation*}
Since $L$ is trivial we conclude that
\begin{equation*}
x_0=\sum_{j=1}^k{n_j\sum_{x \in B_\delta}{z_{j,x}x}}= n_1h_1+\dots+n_kh_k.
\end{equation*}
It follows that $\phi(n) = x_0$, and so $x_0 \in \phi(\psi(Q\cap V) \cap \Z^k)$.

In the other direction suppose that $x_0 \in \phi(\psi(Q\cap V) \cap \Z^k)$ and $v_0 \in Q \cap \Lambda$ is such that $x_0=\phi(\psi(v_0))$.  Then $v_0 \in \Lambda$ and so
\begin{equation*}
v_0=n_1v_1+\dots+n_kv_k\text{ for some } n \in \Z^k,
\end{equation*}
and so
\begin{equation*}
v_0=\sum_{j=1}^k{n_j\sum_{x \in B_\delta}{z_{j,x}R_\Gamma(x)}}.
\end{equation*}
We exponentiate point-wise as before to get that
\begin{equation*}
\exp(2\pi i v_0)= \prod_{j=1}^k{\left(\prod_{x \in B_\delta}{\gamma(x)^{z_{j,x}}}\right)^{n_j}} = \gamma(\sum_{j=1}^k{n_j\sum_{x \in B_\delta}{z_{j,x}x}}) \text{ for all } \gamma \in \Gamma.
\end{equation*}
But $v_0 \in Q$ and so $|1-\exp(2\pi i v_0)| \leq \delta$ for all $\gamma \in \Gamma$ and hence
\begin{equation*}
x_0=\phi(n) = \sum_{j=1}^k{n_jh_j} = \sum_{j=1}^k{n_j\sum_{x \in B_\delta}{z_{j,x}x}} \in B_\delta
\end{equation*}
as required.  The result is proved.
\end{proof}
In light of the start of the proof here it might be more natural to define a centred convex coset progression to be a set of the form $\bigcup{\phi(Q \cap \Z^d)}$ where $\phi:\Z^d \rightarrow G/H$ is a homomorphism, $H \leq G$ and $Q$ is a symmetry convex body in $\R^d$.  This sort of consideration becomes more relevant as one moves to the non-Abelian setting but this is not our concern here.

\section{Concluding remarks}\label{sec.con}

First we should note that Theorem \ref{thm.finalfre} follows immediately from combining Proposition \ref{prop.grow2} and Theorem \ref{thm.rtc}, and all the $\log^{o(1)}K$ terms can be replaced by $(\log \log K)^{O(1)}$ terms for those interested.

It may be worth noting that there are really three different functions in Theorem \ref{thm.gr}; we really show the following.
\begin{theorem}\label{thm.diff}
Suppose that $A \subset G$ has $|A+A| \leq K|A|$.  Then $A$ is $\exp(h(K))$-covered by a $d(K)$-dimensional centred convex coset progression $M$ of size at most $\exp(f(K))|A|$.
\end{theorem}
The quantities $h(K)$, $d(K)$ and $f(K)$ can be traded off between each other to some extent but there is an associated cost.  The precise relationships are a little ad-hoc because they reflect different combinations of our three main examples.  Let us recall these now:
\begin{enumerate}
\item \emph{(Cosets of subgroups)} Suppose that $H$ is a finite subgroup of $G$ and $X$ is an $H$-separated set of $2K+O(1)$ points.  Then letting $A:=X+H$ we have $|A+A| \sim K|A|$.
\item \emph{(Convex progressions)} Suppose that $M$ is a $d$-dimensional convex coset progression.  Then we have seen that $|M+M| \leq \exp(O(d))|M|$.  On the other hand if $A$ is a cube in $\Z^d$ (so that all he side lengths are the same) then in fact $|A+A| \sim 2^d|A|$ so that the doubling of $A$ really is this large.
\item \emph{(Subsets of subgroups)} Suppose that $H$ is a finite subgroup of $G$ and $A$ is a randomly chosen subset of $H$, taking $x \in H$ with probability $1/K$.  Then with high probability $|A| \sim |H|/K$ and $|A+A| \sim |H|$ so that $|A+A| \sim K|A|$. 
\end{enumerate}
Each of these suggests a lower bound on (respectively) $h(K)$, $d(K)$ and $f(K)$, but they do not all give such bounds and there is no one example which forces lower bounds on all of them simultaneously.  This is because of the previously mentioned ability to trade which we shall now explain in a little more depth.  We assume that we are given Theorem \ref{thm.diff} with some functions $h(K),d(K)$ and $f(K)$.
\subsection{Reducing $h(K)$ in exchange for $d(K)$}  One can eliminate $h(K)$ entirely and replace `$\exp(h(K))$-covered by' in Theorem \ref{thm.diff} by `contained in' at the expense of replacing $d(K)$ by $d(K)+\exp(h(K))$, and $f(K)$ by $2f(K)$.  This is a little fiddly, but not difficult to do.

Removing the dependence on covering number is the additional requirement which is made in traditional statements of Fre{\u\i}man-type theorems; indeed, Green and Ruzsa in \cite{greruz::0} actually proved the following.
\begin{theorem}[Green-Ruzsa theorem, original version]\label{thm.grorig}
Suppose that $|A+A| \leq K|A|$.  Then $A$ is  contained in a $K^{4+o(1)}$-dimensional centred convex coset progression $M$ of size at most $\exp(K^{4+o(1)})|A|$.
\end{theorem}
This has been slightly improved, with the power of $4+o(1)$ being replaced by $1+o(1)$ but the reason we do not use this formulation is that the dimension bound must be at least $\Omega(K)$ -- exponentially worse than in the Polynomial Fre{\u\i}man-Ruzsa conjecture.  This is, of course, suggested by the fact that reducing the covering number has a cost of $\exp(h(K))$ rather than $h(K)$ associated with it. 

To see the difficulty directly suppose that $A$ is a set of $2K+O(1)$  generators of a torsion-free group.  Then $|A+A| \sim K|A|$, but any convex coset progression containing $A$ has dimension at least $2K-O(1)$.  

\subsection{Reducing $d(K)$ in exchange for $f(K)$} In general one cannot trade all of the dimension in for size, but one can if the group has bounded exponent (meaning every element has order bounded by an absolute constant).  Then one may reduce $d(K)$ to $0$ at the expense of replacing $f(K)$ by $\exp(f(K)+O(d(K)))$.  In Theorem \ref{thm.finalfre} this gives the following result.
\begin{theorem}\label{thm.frebound}
Suppose that $G$ is a group of bounded exponent and $A\subset G$ has $|A+A| \leq K|A|$.  Then $A$ is $\exp(O(\log^{3+o(1)} K))$-covered by a subgroup $M$ of size at most $\exp(O(\log^{3+o(1)} K))|A|$.
\end{theorem}
Conjecturally one can do much better, and here the Polynomial Fre{\u\i}man-Ruzsa conjecture becomes the following which was one of its (PFR's) original motivations.
\begin{conjecture}[Marton's conjecture]
Suppose that $G$ is a group of bounded exponent and $A\subset G$ has $|A+A| \leq K|A|$.  Then $A$ is $\exp(O(\log K))$-covered by a subgroup $M$ of size at most $\exp(O(\log K))|A|$.
\end{conjecture}

\subsection{Reducing $d(K)$ in exchange for $h(K)$}  We just saw how to trade dimension in for size in the case where the group has bounded exponent.  In general one cannot trade all of the dimension in for size but Green and Tao in \cite{gretao::8} show (in torsion-free groups) how to reduce the dimension of the progression to $O(\log K)$ while incurring an exponential cost in the covering number so that $h(K)=\Theta(K)$.  (They get a larger polynomial in $K$ in their work but this can be removed given the recent stronger bounds in Fre{\u\i}man's theorem.)

The paper \cite{gretao::8} is, in general, rather useful as a source of tools for giving the lower bounds on the order of relative polynomial growth of sets and we direct the reader interested in the more precise relationships between $h(K), d(K)$ and $f(K)$ there.

As a final remark it is worth saying that convex progressions may not be quite the right notion to deal with and one might like to ask for a convex progression of a particular type.  There is some discussion of this in \cite{gretao::8} but we shall not pursue this here, except to remark that Fre{\u\i}man's theorem is usually stated using generalised arithmetic progressions which are a special type of (translate of a centred) convex progression defined by a cube.  Specifically a set $M$ is a \textbf{generalised arithmetic progression} if
\begin{equation*}
M=\{x_0+z_1x_1+\dots+z_dx_d: |l_i| \leq L_i\}
\end{equation*}
for some natural numbers $L_1,\dots,L_d$ and elements $x_0,\dots,x_d \in G$.  If we define a homomorphism
\begin{equation*}
\phi:\Z^d \rightarrow G; (z_1,\dots,z_d) \mapsto z_1x_1+\dots+z_dx_d,
\end{equation*}
and a convex set $Q:=\prod_{i=1}^d{[-L_i,L_i]}$ then $M=x_0+\phi(\Z^d\cap Q)$.  A \textbf{coset progression} (as defined by Green and Ruzsa in \cite{greruz::0}) is then a set of the form $M+H$ where $H \leq G$ and $M$ is a generalised arithmetic progression in $G$.  Proving the results of this paper for coset progressions instead of convex coset progressions is not conceptually harder, but does seem to involve some additional technical difficulties.  

Generalised arithmetic progressions have been studied in there own right and there are various questions concerning whether they are proper or not, meaning whether $\phi$ is injective on $Q \cap \Z^d$.  Bilu in \cite{bil::} has a nice discussion of this (see also \cite[\S3.1]{taovu::}).

\section{Applications}\label{sec.apps}

As indicated in the introduction there are numerous applications of Fre{\u\i}man's theorem, and for completeness we shall discuss a few of these here.  These are mainly chosen because they do not require too much additional material to develop rather than because they are necessarily the most exciting.  This section is of a much more sketchy nature than the rest of the paper: it is intended to indicate directions one can take the results discussed in this paper; it is not intended to cover them in detail and the interested reader is referred to the papers indicated in each subsection below for more comprehensive discussions.

One thing it is worth remembering is that while Fre{\u\i}man's theorem is very attractive at a qualitative level, in applications one can often squeeze a little more juice out of the situation by using the methods of this paper rather than the results.  In particular the combinatorial arguments on their own are often enough for what one hopes to do.  In this regard it should be mentioned that there are many direct applications of the techniques of Croot and Sisask in \cite{crosis::} and \cite{croabasis::}, which can also be proved using Fre{\u\i}man's theorem but which only really require the Croot-Sisask lemma.

A second remark is due with regard to Roth's theorem.  The reader may be hoping for a discussion of bounds in Roth's theorem in this survey, but this is not really the place for that.  In particular, while the results of Proposition \ref{prop.grow1} are relevant to that work, nothing else from the paper is, and a discussion of the combinatorial techniques of Katz and Koester \cite{katkoe::} and the regular Bohr set technology of Bourgain \cite{bou::5} would be required.

\subsection*{The $U^3$-inverse theorem} Gowers' work \cite{gow::4} marks the start of an explosion of applications of Fre{\u\i}man's theorem after he made the crucial observation that it can be combined with the Balog-Szemer{\'e}di lemma \cite{balsze::}.  Gowers used Fre{\u\i}man's theorem to improve the bounds in Szemer{\'e}di's theorem for arithmetic progressions of length four and a little after that Green and Tao expressed Gowers' ideas in a framework often described as `quadratic Fourier analysis'.  Indeed, Gowers' original aim seems to have included finding a proof of Szemer{\'e}di's theorem which was closer to Roth's proof of Roth's theorem for arithmetic progressions of length three and Green and Tao's framework helps highlight these parallels.  This subsection is more thoroughly explained in the paper \cite{gretao::1}.

Roth's proof of Roth's theorem has, at its core, something now called a $U^2$-inverse theorem.  The $U^2$-norm of a function $f$ on a finite (compact) Abelian group $G$ is defined by
\begin{equation*}
\|f\|_{U^2(G)}^4=\E_{x,y,z \in G}{f(x)\overline{f(x+y)f(x+z)}f(x+y+z)}.
\end{equation*}
It turns out that this is a norm and if $A$ and $B$ are two sets in $G$ with $\|1_A - 1_B\|_{U^2(G)}$ small then the number of three-term arithmetic progressions in $A$ is close to that in $B$.  This is why the $U^2$-norm is useful for understanding problems about three-term arithmetic progressions.  It turns out that if a function does not have small $U^2$-norm then it has a linear bias in the following sense.
\begin{theorem}[$U^2(\F_2^n)$-inverse theorem]
Suppose that $f \in L^\infty(\F_2^n)$ has $\|f\|_{U^2(\F_2^n)} \geq \delta \|f\|_{L^\infty(\F_2^n)}$.  Then there is a linear polynomial $l:\F_2^n \rightarrow \F_2$, meaning a map $x \mapsto r \cdot x$ for some $r \in \F_2^n$, such that
\begin{equation*}
|\langle f,(-1)^l\rangle_{L^2(\F_2^n)}| \geq\delta^{O(1)} \|f\|_{L^\infty(\F_2^n)}.
\end{equation*}
\end{theorem}
This is essentially trivial to prove and, a version for the group $G=\Z/N\Z$ rather than $\F_2^n$, can be used as the basis for an iteration to prove Roth's theorem on three-term arithmetic progressions.

Now suppose that one is interested in four-term arithmetic progressions.  In this case if we have two sets $A$ and $B$ with $\|1_A - 1_B\|_{U^2(G)}$ small it is \emph{not} necessarily the case that $A$ and $B$ have similar numbers of four-term arithmetic progressions.  There is, however, a stronger norm called the $U^3$-norm for which this is true.  The $U^3$-norm of a function $f$ on a finite (compact) Abelian group $G$ is defined by
\begin{eqnarray*}
\|f\|_{U^3(G)}^8&=&\E_{x,y,z,w \in G}{\left(f(x)\overline{f(x+y)f(x+z)f(x+w)}\cdot\right.}\\ & & \left. \times {f(x+y+z)f(x+y+w)f(x+z+w)\overline{f(x+y+z+w)}}\right).
\end{eqnarray*}
It turns out that this is also a norm and there is a $U^3$-inverse theorem.  This is where Theorem \ref{thm.finalfre} can be inserted into the various proofs of the inverse theorem.  For $\F_2^n$ this is due to Samorodnitsky \cite{sam::} (see also \cite{wol::}) for $\F_2^n$, and one gets the following.
\begin{theorem}[$U^3(\F_2^n)$-inverse theorem]
Suppose that $f \in L^\infty(\F_2^n)$ has $\|f\|_{U^3(\F_2^n)} \geq \delta \|f\|_{L^\infty(\F_2^n)}$.  Then there is a quadratic polynomial $q:\F_2^n \rightarrow \F_2$, meaning a map $x \mapsto x \cdot Ax$ where $A$ is an upper triangular matrix $\F_2^n \rightarrow \F_2^n$, such that
\begin{equation*}
|\langle f,(-1)^q\rangle_{L^2(\F_2^n)}| \geq \exp(-O(\log^{3+o(1)}\delta^{-1}))\|f\|_{L^\infty(\F_2^n)}.
\end{equation*}
\end{theorem}
This is much harder to prove than the $U^2(\F_2^n)$-inverse theorem and there is actually a close relationship between this and Marton's conjecture.  Indeed, Green and Tao in \cite{gretao::6} and Lovett in \cite{lov::} showed that Marton's conjecture for $\F_2^n$ is equivalent to the following.
\begin{conjecture}[Polynomial $U^3(\F_2^n)$-inverse conjecture]
Suppose that $f \in L^\infty(\F_2^n)$ has $\|f\|_{U^3(\F_2^n)} \geq \delta \|f\|_{L^\infty(\F_2^n)}$.  Then there is a quadratic polynomial $q:\F_2^n \rightarrow \F_2$ such that
\begin{equation*}
|\langle f,(-1)^q\rangle_{L^2(\F_2^n)}| \geq \exp(-O(\log \delta^{-1}))\|f\|_{L^\infty(\F_2^n)}.
\end{equation*}
\end{conjecture}
If true this would bring the $U^3(\F_2^n)$-inverse state of affairs in line with the $U^2$ situation.  

Again, the analogue of the $U^3(\F_2^n)$-inverse theorem for the group $G=\Z/N\Z$ can be used to give a proof of Szemer{\'e}di's theorem for progressions of length four, and, of course, there are higher analogues called $U^k$-norms for longer progressions but again we do not discuss this here.

\subsection*{Long arithmetic progressions in sumsets} The question of finding long arithmetic progressions in sets of integers is one of central interest in additive combinatorics.  The basic question has the following form: suppose that $A_1,\dots,A_k \subset \{1,\dots,N\}$ all have density at least $\alpha$.  How long an arithmetic progression can we guarantee that $A_1+\dots+A_k$ contains?

For one set this is addressed by the notoriously difficult Szemer{\'e}di's theorem \cite{sze::,sze::0} where the best quantitative work is that of Gowers \cite{gow::4,gow::0} (as mentioned in the previous subsection); for two sets the longest progression is much longer with the state of the art due to Green \cite{gre::0} (see also Croot and Sisask \cite{crosis::}); for three sets or more the results get even stronger with the work of Fre{\u\i}man, Halberstam and Ruzsa \cite{frehalruz::}; and finally for eight sets or more, longer again by the recent work of Schoen \cite{sch::1}.

The ideas around theorem \ref{thm.finalfre} (see \cite{san::00}) can be used to give an improvement for four sets or more, and in particular we have the following theorem.
\begin{theorem}
Suppose that $A_1,\dots,A_4 \subset \{1,\dots,N\}$ all have density at least $\alpha$.  Then $A_1+\dots+A_4$ contains an arithmetic progression of length $N^{O(\log^{-O(1)}2\alpha^{-1})}$.
\end{theorem}

\subsection*{$\Lambda(4)$-estimate for the squares} A wonderful conjecture of Rudin \cite{rud::0} asserts that the squares are a $\Lambda(4)$-set.  In symbols this is the following conjecture.
\begin{conjecture}
Suppose that $n_1,\dots,n_k$ are natural numbers.  Then
\begin{equation*}
\int{\left|\sum_{i=1}^k{\exp(2\pi i n_i^2\theta)}\right|^4d\theta} =O(k^{2+o(1)}).
\end{equation*}
\end{conjecture}
Inserting ideas around Theorem \ref{thm.finalfre} (see \cite{san::00}) into the work of \cite{cha::1} (itself developed from an argument of Bourgain in \cite{johlin::}) yield the following result
\begin{theorem}
Suppose that $n_1,\dots,n_k$ are natural numbers.  Then
\begin{equation*}
\int{\left|\sum_{i=1}^k{\exp(2\pi i n_i^2\theta)}\right|^4d\theta} =O(k^3\exp(-\Omega(\log^{\Omega(1)} 2k))).
\end{equation*}
\end{theorem}
This is essentially equivalent to inserting Theorem \ref{thm.finalfre} into the proof of \cite[Theorem 8]{sch::1} and Gowers' \cite{gow::4} version of the Balog-Szemer{\'e}di Lemma \cite{balsze::}.  Of course, this is far form Rudin's conjecture but it is still the best known result at this time.

\subsection*{The Konyagin-{\L}aba theorem} Ideas around Theorem \ref{thm.finalfre} (see \cite{san::00}) inserted into the argument at the end of \cite{sch::1} yield the following quantitative improvement to a result from \cite{konaba::}.
\begin{theorem}[Konyagin-{\L}aba theorem]
Suppose that $A$ is a set of reals and $\alpha \in \R$ is transcendental.  Then 
\begin{equation*}
|A+\alpha.A| = \exp(\Omega(\log^{\Omega(1)}2|A|))|A|.
\end{equation*}
\end{theorem}
What is particularly interesting here is that there is a simple construction which shows that there are arbitrarily large sets $A$ with $|A+\alpha.A| = \exp(O(\sqrt{\log |A|}))|A|$.

\section*{Acknowledgements}

The author should very much like to thank Andrew Granville for a very thorough reading of this paper and supplying a much clearer proof of Proposition \ref{prop.coset}, Sergei Konyagin for a talk on his improvements at the Paul Tur{\'a}n memorial conference 2011, Olof Sisask for directing the author's attention to a better proof of the Marcinkiewicz-Zygmund inequality, and an anonymous referee for a very thorough reading of this paper which has made it immeasurably clearer.

It should also be apparent that the author is heavily influenced by the work of Ben Green, Imre Ruzsa and Terry Tao and this survey would not exist without their numerous insights.  Ben, in particular, has been exceptionally generous with his ideas and conversations.

\bibliographystyle{halpha}

\bibliography{references}

\end{document}